\documentclass[11pt]{article}
\usepackage[margin=1.1in]{geometry}
\usepackage{amsmath,amssymb,amsthm}

\theoremstyle{plain}
\newtheorem{theorem}{Theorem}[section]
\newtheorem{lemma}[theorem]{Lemma}
\newtheorem{proposition}[theorem]{Proposition}
\newtheorem{corollary}[theorem]{Corollary}
\newtheorem{conjecture}[theorem]{Conjecture}
\theoremstyle{definition}
\newtheorem{definition}[theorem]{Definition}
\newtheorem{example}[theorem]{Example}
\theoremstyle{remark}
\newtheorem{remark}[theorem]{Remark}

\newcommand{\HH}{\mathrm{HH}}
\newcommand{\HC}{\mathrm{HC}}
\newcommand{\HP}{\mathrm{HP}}
\newcommand{\Tor}{\operatorname{Tor}}
\newcommand{\gldim}{\operatorname{gl.dim}}
\newcommand{\pd}{\operatorname{pd}}
\newcommand{\rk}{\operatorname{rank}}
\newcommand{\tr}{\operatorname{tr}}
\newcommand{\Id}{\mathrm{Id}}
\newcommand{\rad}{\mathfrak{r}}
\newcommand{\La}{\Lambda}
\newcommand{\CC}{\mathcal{C}}
\newcommand{\Ext}{\operatorname{Ext}}
\newcommand{\HomOp}{\operatorname{Hom}}

\title{Protected corners and a trichotomy\\ for Han's conjecture}
\author{Marco Armenta}
\date{}
\begin{document}
\maketitle

\begin{abstract}
Han's conjecture predicts that a finite-dimensional algebra with eventually vanishing Hochschild homology has finite global dimension. In the tau-Hochschild framework of Cibils, Lanzilotta, Marcos and Solotar, it splits into persistence (Gap A) and survival (Gap B), and a Gap-A failure is already a counterexample. We prove a protected corner theorem bounding Ext at a surviving vertex, settling the Liu-Morin extension conjecture beyond the monomial and special biserial cases. We then establish a trichotomy for Gap-A failures on three strongly connected vertices: the all-infinite case is impossible, and the two-infinite case is completely classified as a mutual dumbbell. 

\end{abstract}

\section{Introduction}

In this paper, we prove that a counterexample to Han's conjecture arising from the failure of persistence, if it exists, has one of two completely described shapes on its minimal number of vertices. In other words, the main result of this work goes as follows: an elementary algebra of infinite global dimension which is not of infinite $+$ global dimension has, on three strongly connected vertices, exactly one or two simple modules of infinite projective dimension; the configuration in which all three are infinite does not exist, and the two-infinite configuration is determined up to a central socle ideal, see Theorem~\ref{thm:trichotomy}. The mechanism behind these statements is a protected corner theorem, which produces self-extensions of a simple module in every degree, see Theorem~\ref{thm:protectedcorner}.

Let $k$ be a field and $\La$ a finite dimensional $k$-algebra with Jacobson radical $\rad$ such that $E=\La/\rad$ is separable. Happel asked in \cite{Happel89} whether the vanishing of Hochschild cohomology $\HH^n(\La)$ in all high degrees forces $\gldim\La<\infty$; Buchweitz, Green, Madsen and Solberg \cite{BGMS} answered negatively with the quantum exterior algebras $\La_q=k\langle x,y\rangle/(x^2,\,yx+qxy,\,y^2)$, $q$ not a root of unity. Han \cite{Han06} then conjectured the homological statement:

\begin{conjecture}[Han]\label{conj:han}
If $\HH_n(\La)=0$ for all $n\gg0$, then $\gldim\La<\infty$.
\end{conjecture}

The conjecture is known for monomial algebras \cite{Han06}, commutative algebras \cite{AVP}, Koszul, cellular and graded local algebras \cite{BerghMadsen}, quantum complete intersections \cite{BerghErdmann}, and is preserved under several gluing operations \cite{CRS}; see the survey \cite{Cruz}.

Recently Cibils, Lanzilotta, Marcos and Solotar \cite{CLMS} introduced the $\tau$-Hochschild (co)homology $\HH^\tau_*(\La)$, computed from the \emph{minimal} projective bimodule resolution of $\La$: in the notation of Section~\ref{sec:conv}, $\HH^\tau_n(\La)$ is the space of \emph{cycles} of the Hochschild chain complex associated with the minimal resolution, while $\HH_n(\La)$ is cycles modulo boundaries \cite[Thm.~2.12, Rem.~2.18]{CLMS}. They call $\La$ of \emph{infinite $+$ global dimension} if there is a pair of vertices $(y,x)$ with $y\La x\neq 0$ and $\Tor^\La_*(k_x,{}_yk)$ nonzero in infinitely many degrees, and prove that this holds if and only if $\HH^\tau_*(\La)$ is infinite \cite[Thm.~5.5]{CLMS}. Their results split Conjecture~\ref{conj:han} into two independent statements:
\begin{itemize}
\item[\textbf{(Gap A)}] \emph{Persistence:} $\gldim\La=\infty$ implies $\La$ is of infinite $+$ global dimension.
\item[\textbf{(Gap B)}] \emph{Survival:} infinite $+$ global dimension implies $\HH_*(\La)$ is infinite.
\end{itemize}
A Gap-A failure is already a counterexample to Conjecture~\ref{conj:han} \cite[Rem.~5.7(1)]{CLMS}, and Gap A alone does not imply the conjecture \cite[Rem.~5.7(3)]{CLMS}; together, Gap A and Gap B are equivalent to it.

Throughout, $\La=kQ/I$ is elementary (a bound quiver algebra) unless stated otherwise, with $v=|Q_0|$ vertices; the general separable case is treated where it costs nothing. We write $S_x$ for the simple module at a vertex $x$ and $\Gamma_x=e_x\La e_x$ for the corner algebra at $x$. We now state the main results of this paper.

\begin{itemize}
\item[\textbf{(A)}] We prove that if $\Gamma_x\neq k$ and $y\La x=0$ for every external in-neighbour $y$ of $x$, that is, no path from $x$ back to any vertex pointing at $x$ survives in $\La$, then the minimal $\Gamma_x$-free resolution of $k$ embeds as a protected thread in the minimal $\La$-resolution of $S_x$, so that
\[
\dim_k\Ext^n_\La(S_x,S_x)\;\geq\;\dim_k\Ext^n_{\Gamma_x}(k,k)\;\geq\;1\qquad\text{for all }n\geq1 ,
\]
see Theorem~\ref{thm:protectedcorner}. In particular the extension conjecture holds at $x$, in a family of instances not covered by the known monomial \cite{GSZ} and special biserial \cite{LiuMorin} cases, and $\La$ is of infinite $+$ global dimension.
\item[\textbf{(B)}] We prove that a Gap-A failure with strongly connected quiver on three vertices, the minimal number by Theorem~\ref{thm:scc} and \cite[\S6]{CLMS}, has exactly one or two simple modules of infinite projective dimension, see Theorem~\ref{thm:trichotomy}. The all-infinite configuration is impossible: the Peirce components that it forces to vanish form precisely the protection hypothesis of (A). The two-infinite configuration is completely determined and we call it the \emph{mutual dumbbell}: two mutually $\La$-dead vertices joined through a hub $u$ whose simple module has finite projective and injective dimension, and, up to a central socle ideal, $\La$ is the fibre product $\La_x\times_k\La_w$ of two two-vertex algebras glued at $u$, see Theorem~\ref{thm:gluing}.
\item[\textbf{(C)}] We prove that any Gap-A failure, on any number of vertices and in every characteristic, satisfies $\HH_n(\La)=0$ for all $n\gg0$ and $\sum_i(-1)^i\dim\HH_i(\La)=v$, see Proposition~\ref{prop:hhvanish}.
\item[\textbf{(D)}] We prove that on the dumbbell the linking bimodules $x\La u$ and $w\La u$ are non-projective over the nontrivial corners and classically generate their singularity categories, see Corollary~\ref{cor:nonproj}; that the Cartan determinant is nonzero and divides five explicit products, see Proposition~\ref{prop:cartanD}; and we construct a Hattori--Stallings transfer through the hub which computes every through-cycle class in $\HH_0(\La)$ by a finite alternating trace sum, although the resolutions of $S_x$ and $S_w$ are infinite, see Lemma~\ref{lem:transfer}.
\end{itemize}

Let us now relate these results to the literature. The trace methods that we use go back to Hattori \cite{Hattori} and Stallings \cite{Stallings}, were brought to finite dimensional algebras by Lenzing \cite{Lenzing} and Igusa \cite{Igusa90}, and culminated in the proof of the strong no loop conjecture by Igusa, Liu and Paquette \cite{ILP}, whose theorem is also an ingredient of Theorem~\ref{thm:trichotomy}. The extension conjecture was posed by Liu and Morin \cite{LiuMorin} and again in \cite{ILP}. Han proposed a proof of it for elementary algebras via recollements and differential graded methods in the preprint \cite{HanExt}, which appears to have remained unpublished, and \cite{CLMS} treat the conjecture as open. The results of this paper are independent of \cite{HanExt}; where both apply, Theorem~\ref{thm:protectedcorner} gives the stronger conclusion of nonvanishing in \emph{every} degree with an explicit lower bound, by an elementary and self-contained argument, see Remark~\ref{rem:hanext}. On the gluing side, we note that the dumbbell is a null-square-type extension in the sense of Cibils, Redondo and Solotar \cite{CRS}, but Corollary~\ref{cor:nonproj} places it exactly outside the null-square \emph{projective} and bounded-extension classes for which Han's conjecture is known to propagate \cite{CRS,CLMSsplit,CLMSbounded,Cruz}. We also note that finite global dimension forbids $2$-truncated cycles \cite{BHM}, and, for monomial algebras, even $m$-truncated ones \cite{Han10} , while the crossing relations of the dumbbell kill \emph{all} through-hub composites. The mixed-necklace computation proposed in Remark~\ref{rem:dumbbellstatus} belongs to the Poincar\'e-series tradition for fibre products of local rings initiated by Dress and Kr\"amer \cite{DressKramer}; see also \cite{NassehSW}. For the Cartan determinant circle around Proposition~\ref{prop:cartanD} see Fuller's survey \cite{Fuller}; for singularity categories see \cite{Buchweitz,Orlov}.

These theorems rest on foundational results that we prove first.
\begin{itemize}
\item[\textbf{(I)}] We prove that $\La$ is of infinite $+$ global dimension if and only if there is a vertex $x$ such that $e_x\,\Omega^n(S_x)\neq 0$ for infinitely many $n$, see Theorem~\ref{thm:criterion}. The engine is the identity $\dim B_n(\La)=\tr S_n+\tr S_{n+1}$ relating the Hochschild chain spaces of the minimal bimodule resolution to the syzygy dimension matrices, see Corollary~\ref{cor:bn}.
\item[\textbf{(II)}] We reduce Gap A to algebras with strongly connected quiver: infinite $+$ global dimension lifts along triangular corner decompositions, and infinite global dimension is detected on the corner algebras of the strongly connected components of $Q$, see Theorem~\ref{thm:scc}. A minimal Gap-A failure therefore has strongly connected quiver. Unconditional new cases follow: Gap A holds whenever every cycle of $Q$ is a loop, in contrast with \cite[Prop.~6.11]{CLMS} no extension conjecture is needed, and whenever every strongly connected component satisfies a per-component criterion, see Corollary~\ref{cor:omnibus}.
\item[\textbf{(III)}] We prove the $\chi$-law: Define on cyclic homology, $\HP_i(\La):=\varprojlim\big(\HC_{i+2m}(\La),S\big)$, for $i=0,1$ and $S$ the morphism in Connes' periodicity exact sequence. If $\HH_n(\La)=0$ for $n>N$ then, over any field,
\[
\sum_{i=0}^{N}(-1)^i\dim_k\HH_i(\La)\;=\;\dim_k\HP_0(\La)-\dim_k\HP_1(\La),
\]
and in characteristic zero the right-hand side equals $\dim_k E/[E,E]$, hence $v$ for elementary algebras, see Theorem~\ref{thm:chi}. Any counterexample to Conjecture~\ref{conj:han} in characteristic zero therefore satisfies the exact budget $\dim\HH_0-v=\sum_{i\geq1}(-1)^{i-1}\dim\HH_i$, see Corollary~\ref{cor:budget}.
\item[\textbf{(IV)}] We prove that every invariant factoring through the dimension data of the minimal bimodule resolution, such as the Cartan matrix, the Ext multiplicities, the syzygy dimension vectors, truncated Euler characteristics and Coxeter traces, is consistent with a counterexample, see Corollary~\ref{cor:nogo}. Gap B is a statement about ranks of differentials, not dimensions of chain spaces.
\end{itemize}

Next, we provide a summary of each Section. In Section~\ref{sec:conv}, we fix notation and prove the $K_0$ identity relating the syzygy dimension matrices to the Ext and Cartan matrices. In Section~\ref{sec:criterion}, we prove the syzygy-return criterion (I). In Section~\ref{sec:gapA}, we prove the reduction (II) of Gap A to strongly connected quivers. In Section~\ref{sec:chi}, we prove the $\chi$-law (III). In Section~\ref{sec:gapB}, we prove the no-go theorem (IV) for Gap B. In Section~\ref{sec:program}, we assemble these results into a stratified program for the search of counterexamples, and we outline the possible directions. In Section~\ref{sec:prob1}, we determine the anatomy of a Gap-A failure: the one-way law, the corner-projectivity criterion and the tiling mechanism. In Section~\ref{sec:threevertex}, we prove the concentration theorem, the tail factorization, the protected corner theorem (A) and the trichotomy (B). In Section~\ref{sec:dumbbell}, we study the dumbbell by means of Hattori--Stallings traces and we prove (C) and (D).

\section{Conventions and the $K_0$ identity}\label{sec:conv}

$k$ is a field; $\La$ is finite dimensional with $E=\La/\rad$ separable. For statements involving quivers, $\La=kQ/I$ with $I$ admissible, $E=kQ_0$, $v=|Q_0|$; $S_x={}_xk$ is the simple left module at the vertex $x$, $P_x=\La e_x$ its projective cover, and $\Omega$ the syzygy operator with respect to minimal projective covers. All matrices below are $v\times v$ with nonnegative integer entries, rows and columns indexed by $Q_0$.

\begin{definition}\label{def:matrices}
\begin{itemize}
    \item (i) The \emph{Cartan matrix}: $C_{xy}:=\dim_k y\La x$; note $C_{xx}\geq 1$.
    \item (ii) The \emph{Ext matrices}: $C^{(n)}_{xy}:=\dim_k\Ext^n_\La(S_x,S_y)$, so that $C^{(0)}=\Id$ and $C^{(1)}$ is the arrow-count matrix. By \cite[Prop.~3.8]{CLMS}, $C^{(n)}_{xy}=\dim_k\Tor_n^\La(k_y,{}_xk)$.
    \item (iii) The \emph{syzygy matrices}: $(S_n)_{xz}:=\dim_k e_z\,\Omega^n(S_x)$, the $z$-component of the dimension vector of $\Omega^n(S_x)$; thus $S_0=\Id$ and $S_1=C-\Id$.
    \item (iv) $E_m:=\sum_{i=0}^m(-1)^iC^{(i)}$ and $\chi_m:=\sum_{i=0}^m(-1)^i\dim_k\HH_i(\La)$.
\end{itemize}
\end{definition}

By Happel's description of the minimal projective bimodule resolution $(P_\bullet,d_\bullet)$ of $\La$ \cite{Happel89}, made explicit in \cite[Thm.~3.2]{CLMS}, one has $P_n\cong\La\otimes_E\Tor_n^\La(E,E)\otimes_E\La$, and the chain spaces of $\La\otimes_{\La^e}P_\bullet$, whose homology is $\HH_*(\La)$, are
\begin{equation}\label{eq:Bn}
B_n(\La)=\bigoplus_{y,x\in Q_0} y\La x\otimes \Tor_n^\La(k_x,{}_yk),
\qquad
b_n:=\dim_k B_n(\La)=\tr\!\big(C^{(n)}C\big),
\end{equation}
using $\dim\Tor_n(k_x,{}_yk)=C^{(n)}_{yx}$ and $\dim y\La x=C_{xy}$; see \cite[(4.7)]{CLMS}. Dually, the cochain spaces have dimensions $a_n:=\dim\HomOp_{E^e}(\Tor_n^\La(E,E),\La)=\tr(C^{(n)}C^{\mathsf T})$ \cite[(4.6)]{CLMS}.

\begin{lemma}[$K_0$ identity]\label{lem:K0}
For all $n\geq0$,
\[
S_n+S_{n+1}=C^{(n)}\,C,
\qquad\text{hence}\qquad
E_m\,C=\Id+(-1)^m S_{m+1}\quad\text{for all }m\geq0 .
\]
\end{lemma}

\begin{proof}
Minimality gives $\operatorname{top}\Omega^n(S_x)\cong\bigoplus_y S_y^{\,C^{(n)}_{xy}}$, since $\dim\HomOp_\La(\Omega^nS_x,S_y)=\dim\Ext^n_\La(S_x,S_y)$. Reading the exact sequence
$0\to\Omega^{n+1}(S_x)\to\bigoplus_y P_y^{\,C^{(n)}_{xy}}\to\Omega^n(S_x)\to0$
at the $z$-th Peirce component, with $\dim e_zP_y=\dim z\La y=C_{yz}$, yields
$(S_n)_{xz}+(S_{n+1})_{xz}=\sum_y C^{(n)}_{xy}C_{yz}$. The second identity follows by telescoping, using $S_0=\Id$.
\end{proof}

\begin{corollary}\label{cor:bn}
$b_n=\tr S_n+\tr S_{n+1}$ for all $n\geq0$. If $\gldim\La=d<\infty$ then $S_{d+1}=0$, so $\sum_{n=0}^d(-1)^nC^{(n)}=C^{-1}$, in particular $C\in\mathrm{GL}_v(\mathbb{Z})$, and
\[
\sum_n(-1)^n b_n=\tr(\Id)=v,
\qquad
\sum_n(-1)^n a_n=\tr\!\big(C^{-1}C^{\mathsf T}\big).
\]
Since the complexes are finite, the left identity together with the theorem of Han and Keller that $\HH_n(\La)=0$ for $n\geq1$ when $\gldim\La<\infty$ (see \cite[\S3]{CLMS}) gives $\dim\HH_0(\La)=v$; the right identity is Happel's trace formula $\sum(-1)^n\dim\HH^n(\La)=\tr(C^{-1}C^{\mathsf T})$ \cite{Happel97}.
\end{corollary}

\begin{proof}
Take the trace of the first identity of Lemma~\ref{lem:K0} for the first claim; the rest is immediate from $\tr(E_dC)=\tr\Id$ and $\tr(E_dC^{\mathsf T})=\tr(C^{-1}C^{\mathsf T})$, since the Euler characteristic of a finite complex equals that of its homology.
\end{proof}

\section{The syzygy-return criterion}\label{sec:criterion}

\begin{theorem}\label{thm:criterion}
For an elementary algebra $\La=kQ/I$ the following are equivalent:
\begin{enumerate}
\item $\La$ is of infinite $+$ global dimension in the sense of \cite[Def.~5.4]{CLMS};
\item $\HH^\tau_*(\La)$ is infinite;
\item $b_n\neq0$ for infinitely many $n$;
\item $\tr S_n\neq0$ for infinitely many $n$;
\item there is a vertex $x\in Q_0$ with $e_x\,\Omega^n(S_x)\neq0$ for infinitely many $n$.
\end{enumerate}
\end{theorem}

\begin{proof}
$(1)\Leftrightarrow(2)$ is \cite[Thm.~5.5]{CLMS}. $(2)\Leftrightarrow(3)$: by \cite[Thm.~5.3]{CLMS}, $\HH^\tau_i(\La)=0$ for all $i\geq N$ if and only if $B_i(\La)=0$ for all $i>N$; hence $\HH^\tau_*$ is infinite iff $b_n\neq0$ infinitely often. $(3)\Leftrightarrow(4)$: by Corollary~\ref{cor:bn}, $b_n=\tr S_n+\tr S_{n+1}$ with nonnegative summands, so $b_n\neq 0$ infinitely often iff $\tr S_n\neq0$ infinitely often. $(4)\Leftrightarrow(5)$: $\tr S_n=\sum_x(S_n)_{xx}$ has finitely many summands.
\end{proof}

\begin{remark}\label{rem:extconj}
Since $\operatorname{top}\Omega^n(S_x)$ is a quotient of $\Omega^n(S_x)$, one has $(S_n)_{xx}\geq C^{(n)}_{xx}=\dim\Ext^n(S_x,S_x)$. Condition (5), \emph{dimension return} of syzygies, is therefore implied by, and strictly weaker than, the conclusion of the extension conjecture (\emph{top return}: $\Ext^n(S_x,S_x)\neq0$ infinitely often at loop vertices; see \cite[\S6.4]{CLMS}). This weakening is what makes the unconditional results of Section~\ref{sec:gapA} possible.
\end{remark}

\begin{remark}
Condition (5) involves only dimension vectors of syzygies of simple modules, data that behave functorially under the idempotent-cutting operations of the next section. This robustness is the design principle of the reduction.
\end{remark}

\section{Reduction to strongly connected quivers}\label{sec:gapA}

\subsection{Corner cutting}

Let $1=e+f$ be a sum of orthogonal idempotents of $\La$ (sums of vertex idempotents in the elementary case) with
\begin{equation}\label{eq:triangular}
f\La e=0 .
\end{equation}
Set $A:=e\La e$, $B:=f\La f$, $M:=e\La f$; then $\La\cong\begin{pmatrix}A&M\\0&B\end{pmatrix}$ is a triangular matrix algebra. In the elementary case \eqref{eq:triangular} holds iff there is no path in $Q$ from an $e$-vertex to an $f$-vertex, i.e.\ the $f$-vertex set is closed under predecessors.

\begin{lemma}[Corner cutting]\label{lem:cut}
Assume \eqref{eq:triangular} with $\La$ elementary; write $a,a'$ for $e$-vertices and $b,b'$ for $f$-vertices.
\begin{enumerate}
\item For every $e$-vertex $a$: $P_\La(a)=\La e_a=Ae_a$, and the entire minimal $\La$-resolution of $S_a$ consists of $A$-modules and coincides with the minimal $A$-resolution. In particular $\pd_\La S_a=\pd_A S_a$, and $(S^\La_n)_{az}=(S^A_n)_{az}$ for every $e$-vertex $z$, while $(S^\La_n)_{az}=0$ for every $f$-vertex $z$.
\item For every $f$-vertex $b$: multiplication by $f$ is exact, $f\,P_\La(a)=0$, and $f\,P_\La(b')=P_B(b')$. Applying $f\cdot(-)$ to the minimal $\La$-resolution of $S_b$ therefore yields a (possibly non-minimal) $B$-projective resolution of $S_b$; consequently
\[
f\,\Omega^n_\La(S_b)\;\cong\;\Omega^n_B(S_b)\oplus L,
\qquad
(S^\La_n)_{bb'}\;\geq\;(S^B_n)_{bb'} ,
\]
for a projective $B$-module $L$, and $\pd_\La S_b\geq\pd_B S_b$.
\end{enumerate}
\end{lemma}

\begin{proof}
(1) $f\La e_a\subseteq f\La e=0$, so $\La e_a=e\La e_a=Ae_a$. Since $\rad_\La=\left(\begin{smallmatrix}\rad_A&M\\0&\rad_B\end{smallmatrix}\right)$ acts on the column $(Ae_a;0)$ through $\rad_A$, radicals, tops and projective covers of modules with zero $f$-part are computed in $A$; syzygies stay in this subcategory, which is equivalent to $A$-mod.

(2) $f\cdot(-)$ is exact because $f$ is idempotent. $f\La e_a=0$ by Eq.~\eqref{eq:triangular}; and $f\La e_{b'}=Be_{b'}$ because a path starting at $b'$ that enters the $f$-part cannot have left it, again by Eq.~\eqref{eq:triangular}. Applying $f$ to the minimal resolution $Q_\bullet\to S_b$ deletes the projectives at $e$-vertices and sends those at $f$-vertices to the corresponding $B$-projectives, giving an exact complex of projective $B$-modules augmented to $fS_b=S_b$. Any projective resolution is the direct sum of the minimal one and a split exact complex of projectives; taking kernels, $f\Omega^n_\La(S_b)\cong\Omega^n_B(S_b)\oplus L$. Since $e_{b'}f=e_{b'}$, the dimension inequality follows.
\end{proof}

\begin{lemma}\label{lem:gldim}
With the hypotheses of Lemma~\ref{lem:cut}: $\gldim\La=\infty$ if and only if $\gldim A=\infty$ or $\gldim B=\infty$.
\end{lemma}

\begin{proof}
Lemma~\ref{lem:cut} gives $\gldim\La\geq\max(\gldim A,\gldim B)$. The bound $\gldim\La\leq\gldim A+\gldim B+1$ for triangular matrix algebras is classical; see \cite[Ch.~4]{FGR}.
\end{proof}

\begin{proposition}[Lifting]\label{prop:lift}
Assume Eq.~\eqref{eq:triangular}. If $A$ or $B$ is of infinite $+$ global dimension, then so is $\La$. Consequently, if $A$ and $B$ both satisfy Gap A, then $\La$ satisfies Gap A.
\end{proposition}

\begin{proof}
If $A$ is infinite $+$, Theorem~\ref{thm:criterion}(5) provides an $e$-vertex $a$ with $(S^A_n)_{aa}\neq0$ for infinitely many $n$; by Lemma~\ref{lem:cut}(1), $(S^\La_n)_{aa}=(S^A_n)_{aa}$, and Theorem~\ref{thm:criterion} applies to $\La$. If $B$ is infinite $+$, Lemma~\ref{lem:cut}(2) gives $(S^\La_n)_{bb}\geq(S^B_n)_{bb}\neq0$ infinitely often. For the last claim: if $\gldim\La=\infty$, then by Lemma~\ref{lem:gldim} one of the corners has infinite global dimension, hence is infinite $+$ by hypothesis, hence $\La$ is infinite $+$.
\end{proof}

\subsection{Strongly connected components}

Let $\CC_1,\dots,\CC_r$ be the strongly connected components (SCCs) of $Q$, topologically ordered so that arrows go from $\CC_i$ to $\CC_j$ only if $i\leq j$. For an SCC $\CC$ put $\varepsilon_\CC=\sum_{x\in\CC}e_x$ and $\La_\CC:=\varepsilon_\CC\,\La\,\varepsilon_\CC$.

\begin{lemma}\label{lem:scc-corner}
Every path in $Q$ between two vertices of $\CC$ stays inside $\CC$. Consequently $\La_\CC\cong kQ_\CC/I_\CC$ for the full subquiver $Q_\CC$ on $\CC$ with the induced relations, independently of the rest of $Q$.
\end{lemma}

\begin{proof}
A path leaving $\CC$ enters a strictly later component and can never return, since returning would merge the components.
\end{proof}

\begin{theorem}[SCC reduction]\label{thm:scc}
Let $\La=kQ/I$ be elementary.
\begin{enumerate}
\item If $\gldim\La=\infty$, then $\gldim\La_\CC=\infty$ for some SCC $\CC$ of $Q$.
\item If some $\La_\CC$ is of infinite $+$ global dimension, then so is $\La$.
\item Consequently, if every SCC-corner $\La_\CC$ with $\gldim\La_\CC=\infty$ is of infinite $+$ global dimension, then $\La$ satisfies Gap A. In particular, Gap A holds for all elementary algebras if and only if it holds for all elementary algebras with strongly connected quiver, and a minimal Gap-A failure has strongly connected quiver.
\end{enumerate}
The same statements hold with ``infinite co$+$'' in place of ``infinite $+$'' (Proposition~\ref{prop:op}).
\end{theorem}

\begin{proof}
Peel off the first component: $f:=\varepsilon_{\CC_1}$ satisfies Eq.~\eqref{eq:triangular} with $e=1-f$, because $\CC_1$ receives no arrows from outside itself. Here $B=\La_{\CC_1}$, and since no path between $e$-vertices can traverse $\CC_1$, the corner $A=e\La e$ is the bound quiver algebra on $Q\setminus\CC_1$ with induced relations, whose SCCs are $\CC_2,\dots,\CC_r$ with unchanged corner algebras (Lemma~\ref{lem:scc-corner}).

(1) By Lemma~\ref{lem:gldim}, $\gldim\La=\infty$ forces $\gldim\La_{\CC_1}=\infty$, and we are done, or $\gldim A=\infty$; induct on the number of components.

(2) If $\CC=\CC_1$, apply Proposition~\ref{prop:lift}. Otherwise, by induction $A$ is infinite $+$; by Lemma~\ref{lem:cut}(1) the relevant diagonal syzygy entries of $A$ equal those of $\La$, so $\La$ is infinite $+$ by Theorem~\ref{thm:criterion}(5).

(3) Combine (1) and (2). For the minimality statement: if $\La$ fails Gap A, by (1) some $\La_\CC$ has infinite global dimension; were $\La_\CC$ infinite $+$, (2) would make $\La$ infinite $+$, a contradiction; hence $\La_\CC$ itself fails Gap A, on a strongly connected quiver.
\end{proof}

\begin{corollary}[Per-component criteria]\label{cor:omnibus}
Gap A holds for $\La=kQ/I$ as soon as each strongly connected component $\CC$ of $Q$ satisfies at least one of:
\begin{enumerate}
\item $|\CC|=1$ (with or without loops);
\item $|\CC|=2$;
\item all Peirce components of $\La_\CC$ are nonzero: $y\La x\neq0$ for all $x,y\in\CC$;
\item $\La_\CC$ is $n$-Koszul, or the Yoneda category of $\La_\CC$ admits a finitely generated infinite dimensional subcategory.
\end{enumerate}
In particular, Gap A holds unconditionally whenever every cycle of $Q$ is a loop, and whenever every strongly connected component of $Q$ has at most two vertices.
\end{corollary}

\begin{proof}
By Theorem~\ref{thm:scc}(3) it suffices to check that each $\La_\CC$ of infinite global dimension is infinite $+$.
(1) If $\CC=\{x\}$ then $\La_\CC=e_x\La e_x$ is local; if $\gldim\La_\CC=\infty$ it is not $k$, i.e.\ nontrivial local, and the minimal resolution of its unique simple never terminates, with $(S_n)_{xx}=\dim\Omega^n\neq0$ for all $n$; this is \cite[Prop.~6.1]{CLMS} in the form of Theorem~\ref{thm:criterion}(5).
(2) If $|\CC|=2$ and $Q_\CC$ is strongly connected, there are arrows in both directions, so all four Peirce components of $\La_\CC$ are nonzero (the diagonal ones contain the idempotents, the off-diagonal ones contain arrows), and \cite[Prop.~6.8]{CLMS} applies.
(3) is \cite[Prop.~6.8]{CLMS}; (4) is \cite[Cor.~6.6]{CLMS} and \cite[Thm.~6.5]{CLMS}.
The final sentence: ``every cycle of $Q$ is a loop'' means all SCCs are singletons.
\end{proof}

\begin{proposition}[Opposite side]\label{prop:op}
$\La$ is of infinite co$+$ global dimension if and only if $\La^{\mathrm{op}}$ is of infinite $+$ global dimension. Since $Q_{\La^{\mathrm{op}}}=Q^{\mathrm{op}}$ has the same SCC structure and $\gldim\La^{\mathrm{op}}=\gldim\La$, all results of this section hold verbatim for infinite co$+$ global dimension, hence also for the $\tau$-version of Happel's question \cite[\S5]{CLMS}.
\end{proposition}

\begin{proof}
Under the standard identifications, $y\La^{\mathrm{op}}x=x\La y$ and $\Tor^{\La^{\mathrm{op}}}_n(k_x,{}_yk)\cong\Tor^\La_n(k_y,{}_xk)$; comparing \cite[Def.~5.4]{CLMS} for $\La$ and for $\La^{\mathrm{op}}$ gives the equivalence.
\end{proof}

\begin{remark}[Comparison with CLMS~\cite{CLMS}]\label{rem:compare}
Corollary~\ref{cor:omnibus} strictly extends the unconditional reach of \cite[\S6]{CLMS}: Prop.~6.11 there obtains infinite $\pm$ global dimension from a loop only under the extension conjecture at the loop vertex, whereas the singleton-SCC case above is unconditional. The mechanism is Remark~\ref{rem:extconj}: the extension conjecture asserts \emph{top return} of syzygies at the loop vertex, while infinite $+$ only requires \emph{dimension return}, which the local corner supplies for free. This does \emph{not} prove the extension conjecture; nor does it cover a loop sitting inside a larger strongly connected component (Remark~\ref{prob:loop}). Note also that \cite[Ex.~6.10]{CLMS} shows the criteria of Corollary~\ref{cor:omnibus} are genuinely independent of Yoneda finite generation.
\end{remark}

\begin{remark}[Novelty caveat]\label{rem:novelty-gapA}
The results of this section are elementary consequences of \cite{CLMS} together with classical homological algebra of triangular matrix rings; a systematic check against the extension series of Cibils, Lanzilotta, Marcos, Redondo and Solotar on null-square and bounded extensions \cite{CRS,CLMSsplit,CLMSbounded} has not located them there, where closure of the class of algebras satisfying Han's conjecture under related gluings is proved via Jacobi--Zariski sequences, and against the recollement-theoretic reductions of Han's conjecture to derived $2$-simple rings (Wang--Xu--Zhang--Zhou \cite{WXZZ}; see also the survey \cite{Cruz}). Those reductions operate on the full conjecture and use derived-category machinery; the present reduction is Morita-level, applies to Gap A alone, and is compatible with the syzygy-return criterion, which is what the search program of Section~\ref{sec:program} requires.
\end{remark}

\section{The $\chi$-law}\label{sec:chi}

In this section, $\La$ is any finite dimensional $k$-algebra. We use cyclic homology $\HC_*(\La)$ and Connes' periodicity exact sequence
\begin{equation}\label{eq:SBI}
\cdots\longrightarrow \HH_n(\La)\xrightarrow{\;I\;}\HC_n(\La)\xrightarrow{\;S\;}\HC_{n-2}(\La)\xrightarrow{\;B\;}\HH_{n-1}(\La)\xrightarrow{\;I\;}\HC_{n-1}(\La)\longrightarrow\cdots
\end{equation}
valid over any field \cite[2.2.1]{Loday}, with $\HC_{-1}=\HC_{-2}=0$. Since $\La$ is finite dimensional, each $\HC_n(\La)$ is finite dimensional (the cyclic bicomplex in total degree $n$ involves only the finite dimensional spaces $\La^{\otimes m}$, $m\leq n+1$). Define $\HP_i(\La):=\varprojlim\big(\HC_{i+2m}(\La),S\big)$ for $i=0,1$; whenever $S$ is eventually an isomorphism the limit is attained at any large stage and no $\varprojlim^1$ issues arise.

\begin{lemma}[Bookkeeping]\label{lem:book}
Write $h_j=\dim\HH_j(\La)$, $c_j=\dim\HC_j(\La)$, and set
\[
i_j:=\rk\big(I\colon\HH_j\to\HC_j\big),\quad
s_j:=\rk\big(S\colon\HC_j\to\HC_{j-2}\big),\quad
\beta_j:=\rk\big(B\colon\HC_j\to\HH_{j+1}\big),
\]
with $s_0=s_1=0$ and $\beta_{-1}:=0$. Then exactness of \eqref{eq:SBI} gives, for all $j\geq0$,
\begin{equation}\label{eq:threerel}
h_j=i_j+\beta_{j-1},\qquad c_j=i_j+s_j,\qquad c_j=s_{j+2}+\beta_j ,
\end{equation}
and consequently, for every $M\geq0$,
\[
\chi_M:=\sum_{j=0}^M(-1)^jh_j\;=\;(-1)^M\big(s_{M+2}-s_{M+1}+\beta_M\big).
\]
\end{lemma}

\begin{proof}
In \eqref{eq:SBI} each $\HC_j$ occurs twice: once with incoming $I$ and outgoing $S$, once with incoming $S$ and outgoing $B$. Exactness at $\HH_j$ gives $\ker I_j=\operatorname{im}(B\colon\HC_{j-1}\to\HH_j)$, whence $h_j=i_j+\beta_{j-1}$; exactness at the two occurrences of $\HC_j$ gives $\ker S_j=\operatorname{im}I_j$ and $\ker(B\colon\HC_j\to\HH_{j+1})=\operatorname{im}S_{j+2}$, whence the other two relations. From \eqref{eq:threerel}, $i_j=s_{j+2}+\beta_j-s_j$, so
\[
\chi_M=\sum_{j=0}^M(-1)^j\big(s_{j+2}-s_j\big)+\sum_{j=0}^M(-1)^j\beta_j+\sum_{j=0}^M(-1)^j\beta_{j-1}.
\]
The first sum telescopes to $(-1)^M(s_{M+2}-s_{M+1})$ using $s_0=s_1=0$; the last equals $-\sum_{j=0}^{M-1}(-1)^j\beta_j$, so the $\beta$-terms collapse to $(-1)^M\beta_M$.
\end{proof}

\begin{theorem}[$\chi$-law]\label{thm:chi}
Let $\La$ be a finite dimensional $k$-algebra with $\HH_n(\La)=0$ for all $n>N$.
\begin{enumerate}
\item Over any field: $S\colon\HC_n(\La)\to\HC_{n-2}(\La)$ is an isomorphism for all $n\geq N+2$; hence $\dim\HC_n$ depends only on the parity of $n$ for $n\geq N$, the towers defining $\HP_i$ stabilize, and
\[
\chi(\HH_*):=\sum_{i=0}^{N}(-1)^i\dim_k\HH_i(\La)\;=\;\dim_k\HP_0(\La)-\dim_k\HP_1(\La).
\]
\item If moreover $\operatorname{char}k=0$, then $\HP_*(\La)\cong\HP_*(E)$ for $E=\La/\rad$, and
\[
\chi(\HH_*)=\dim_k E/[E,E].
\]
If $\La$ is elementary, $\chi(\HH_*)=v=|Q_0|$; more generally, if $E$ is split, $\chi(\HH_*)$ equals the number of isomorphism classes of simple $\La$-modules.
\end{enumerate}
\end{theorem}

\begin{proof}
(1) For $n>N$ we have $h_n=0$, hence $i_n=0$ and, by the second relation of \eqref{eq:threerel}, $S_n$ is injective with $s_n=c_n$. For $n\geq N+2$ also $h_{n-1}=0$, hence $\beta_{n-2}=0$ and, by the third relation, $s_n=c_{n-2}$; so $S_n$ is bijective and $c_n=c_{n-2}$. The towers $(\HC_{i+2m},S)$ are therefore eventually constant, with $\HP_i(\La)\cong\HC_{i+2m}(\La)$ for $i+2m\geq N$.

Now take $M\geq N+1$ in Lemma~\ref{lem:book}. Then $\beta_M=0$ (the target $\HH_{M+1}$ vanishes), and since $M+1\geq N+2$ we may substitute $s_{M+2}=c_M$ and $s_{M+1}=c_{M-1}$, obtaining $\chi_M=(-1)^M(c_M-c_{M-1})$. As $h_j=0$ for $j>N$, $\chi_M=\chi_N=\chi(\HH_*)$ for all such $M$; choosing $M$ even and large gives $\chi(\HH_*)=c_M-c_{M-1}=\dim\HP_0-\dim\HP_1$ (an odd $M$ gives the same value).

(2) By Goodwillie's theorem, periodic cyclic homology in characteristic zero is invariant under nilpotent extensions \cite{Goodwillie}, \cite[Thm.~4.1.15]{Loday}; applied to the nilpotent ideal $\rad$, the projection $\La\to E$ induces $\HP_*(\La)\cong\HP_*(E)$. In characteristic zero $k$ is perfect, so $E$ is separable; hence $E$ is projective as an $E$-bimodule and $\HH_n(E)=0$ for $n\geq1$, while $\HH_0(E)=E/[E,E]$. Applying part (1) to $E$ with $N=0$: $S$ is an isomorphism on $\HC_n(E)$ for $n\geq2$; moreover $\HC_0(E)=\HH_0(E)$ always, and exactness of $\HH_1(E)\to\HC_1(E)\xrightarrow{S}\HC_{-1}=0$ with $\HH_1(E)=0$ gives $\HC_1(E)=0$. Hence $\HP_0(E)=E/[E,E]$ and $\HP_1(E)=0$. For elementary $\La$, $E\cong k^v$; for split semisimple $E\cong\prod_i M_{n_i}(k)$, the trace maps identify $E/[E,E]\cong k^{\#\{\text{blocks}\}}$.
\end{proof}

\begin{remark}
Part (1) uses no characteristic assumption; in characteristic $p$ the conclusion is $\chi(\HH_*)=\dim\HP_0(\La)-\dim\HP_1(\La)$, with $\HP_*(\La)$ no longer computable from $E$ by Goodwillie's theorem. Part (2) recovers, for $\gldim\La<\infty$, the identity $\dim\HH_0=v$ of Corollary~\ref{cor:bn}, which, however, held over any field. The genuinely new reach of the theorem is the extension from finite global dimension to finite Hochschild homological dimension.
\end{remark}

\begin{lemma}[Structure of $\HH_0$ of an elementary algebra]\label{lem:HH0}
Let $\La=kQ/I$. Call an element of $y\La x$ with $x\neq y$ \emph{open}, and an element of $x\rad x$ \emph{closed radical}.
\begin{enumerate}
\item Every open element lies in $[\La,\La]$: for $p\in y\La x$ with $x\neq y$, $p=[e_y,p]$.
\item For each vertex $x$, let $\lambda_x\colon\La\to k$ send $a$ to the coefficient of $e_x$ in the $x\La x$-component of $a$, with respect to the decomposition $x\La x=ke_x\oplus x\rad x$. Then $\lambda_x$ vanishes on $[\La,\La]$; consequently the classes $[e_x]\in\HH_0(\La)$ are linearly independent and $\dim\HH_0(\La)\geq v$.
\item $\dim\HH_0(\La)=v$ if and only if every closed radical element, equivalently, the class of every closed path of positive length, lies in $[\La,\La]$.
\end{enumerate}
\end{lemma}

\begin{proof}
(1) $e_yp=p$ and $pe_y=0$ since $x\neq y$.
(2) For $a,b\in\La$ the $x\La x$-component of $ab$ is $\sum_y(xay)(ybx)$. For $y\neq x$ both factors lie in $\rad$, so the product lies in $x\rad x$ and contributes $0$ to $\lambda_x$. For $y=x$, writing $xax=\lambda_x(a)e_x+\rho$, $xbx=\lambda_x(b)e_x+\rho'$ with $\rho,\rho'\in x\rad x$, the product is $\lambda_x(a)\lambda_x(b)e_x$ plus terms in $x\rad x$. Hence $\lambda_x(ab)=\lambda_x(a)\lambda_x(b)=\lambda_x(ba)$, so $\lambda_x([a,b])=0$. Since $\lambda_y(e_x)=\delta_{xy}$, the classes $[e_x]$ are independent in $\La/[\La,\La]$.
(3) $\La=\bigoplus_x ke_x\oplus\bigoplus_x x\rad x\oplus(\text{open part})$, and the open part lies in $[\La,\La]$ by part (1); so $\HH_0$ is spanned by the $[e_x]$ and the classes of closed radical elements. If all the latter vanish, $\dim\HH_0=v$ by part (2). Conversely, if $\dim\HH_0=v$, then for a closed radical $p$ we have $[p]=\sum_x\alpha_x[e_x]$; applying $\lambda_y$, which kills $[\La,\La]$ and $x\rad x$, gives $\alpha_y=0$ for all $y$, so $p\in[\La,\La]$.
\end{proof}

\begin{corollary}[Budget]\label{cor:budget}
Let $\operatorname{char}k=0$ and $\La=kQ/I$ with $\HH_n(\La)=0$ for all $n>N$. Then:
\begin{enumerate}
\item $\displaystyle\dim\HH_0(\La)-v=\sum_{i=1}^{N}(-1)^{i-1}\dim\HH_i(\La)$ \quad (the \emph{$\chi$-budget});
\item if $N=0$, i.e.\ $\HH_i(\La)=0$ for all $i\geq1$, then every closed path of positive length is a sum of commutators; for a local algebra this reads $\rad=[\La,\La]$;
\item in particular, a characteristic-zero counterexample to Conjecture~\ref{conj:han} must repay the excess $\dim\La/[\La,\La]-v$ exactly in its low-degree homology, and if it has no positive-degree homology at all it must sit on the ``Lenzing boundary'': the commutator condition that finite global dimension forces (cf.\ \cite{Lenzing} and Corollary~\ref{cor:bn}) holds while $\gldim=\infty$.
\end{enumerate}
\end{corollary}

\begin{proof}
(1) is Theorem~\ref{thm:chi}(2) rearranged. (2) combines (1) with Lemma~\ref{lem:HH0}(3); for local $\La$ every radical element is closed and $[\La,\La]\subseteq\rad$ by Lemma~\ref{lem:HH0}(2), giving equality. (3) is a restatement.
\end{proof}

\begin{remark}[Necklaces, and the empty local Lenzing boundary]\label{rem:necklace}
The classes of Lemma~\ref{lem:HH0} admit the classical \emph{necklace} description: for $\La=kQ/I$ elementary, $\HH_0(\La)=\La/[\La,\La]$ has $k$-basis the images of the vertices $e_x$ together with the surviving positive-length necklaces, the cyclic-rotation classes of closed paths of $Q$, taken modulo $I$. Hence $\rad=[\La,\La]\iff\dim\HH_0(\La)=v\iff$ no positive-length necklace survives, and if every relation lies in $\rad^{t+1}$ then $Q$ has no oriented closed path of length $\leq t$ surviving in $\La$. One caution is worth recording: for a \emph{monomial} algebra ``no surviving necklace'' does \emph{not} force $Q$ acyclic, $k\mathbb Z_3/\rad^2$ is monomial with $\dim\HH_0=3=v$ (so $\rad=[\La,\La]$), every closed path having length $\geq3>1$, yet its quiver is a $3$-cycle and $\gldim=\infty$. On the \emph{local} locus the situation is sharper: there $[\La,\La]=[\rad,\rad]\subseteq\rad^2$ (scalars are central, so every commutator is a commutator of radical elements). Consequently the commutator condition $\rad=[\La,\La]$ of Corollary~\ref{cor:budget}(2) would force $\rad=\rad^2$, hence $\rad=0$ by Nakayama; a \emph{nontrivial} local algebra therefore never satisfies it, and the ``Lenzing boundary'' of Corollary~\ref{cor:budget}(3) is \emph{empty} on the local locus. A local characteristic-zero counterexample with no positive-degree Hochschild homology thus cannot exist, consistently with the local campaign being a pure Gap-B question (Section~\ref{sec:program}).
\end{remark}

\begin{remark}[Euler-level asymmetry and the Coxeter transformation]\label{rem:coxeter}
For $\gldim\La<\infty$, Corollary~\ref{cor:bn} shows that the homology Euler characteristic is the trivial trace $v$, while the cohomology Euler characteristic is Happel's Coxeter trace $\tr(C^{-1}C^{\mathsf T})$ \cite{Happel97}. Theorem~\ref{thm:chi} extends the homology statement from $\gldim<\infty$ to $\mathrm{hh.dim}<\infty$: \emph{at the level of Euler characteristics, Hochschild homology cannot distinguish a potential counterexample from a smooth algebra}. No $(\La,\La)$-coefficient cohomological analogue can hold: cyclic homology pairs with $H^*(\La,D\La)$, not with $\HH^*(\La)$, and already for the Kronecker algebra $\tr(C^{-1}C^{\mathsf T})=-2\neq2=v$. This Euler-level asymmetry is a shadow of the asymmetry between Han's conjecture and Happel's question. For symmetric algebras $D\La\cong\La$, so $\dim\HH^n=\dim\HH_n$ and the $\chi$-law does constrain cohomology. Since dimensions cannot separate smooth from non-smooth on the homology side, the remaining information sits at the operator level, the calculus structure on $(\HH^*,\HH_*)$, its derived invariance \cite{ArmentaKeller}, and the Coxeter transformation acting on it \cite{Armenta}.
\end{remark}

\begin{remark}[Novelty caveat]\label{rem:novelty-chi}
The proof of Theorem~\ref{thm:chi} assembles standard cyclic-homology facts; statements of this type may well be known. Closely related material appears in Igusa's work relating cyclic homology to the Cartan determinant \cite{Igusa92} and in the Hirzebruch--Riemann--Roch formalism for finite dimensional algebras \cite{HanHRR}. The closest statements that we have located are the graded results of \cite{Igusa92,BerghMadsen} and the Euler formalism of \cite{HanHRR}, neither of which contains Theorem~\ref{thm:chi} in the ungraded generality used here; Lemma~\ref{lem:HH0} is folklore.
\end{remark}

\section{Gap B and dimension invariants}\label{sec:gapB}

Recall from \cite[Thm.~2.12, Rem.~2.18]{CLMS} that, computing with the minimal bimodule resolution, $\HH^\tau_n(\La)=\ker\delta'_n$ is the space of cycles of the chain complex $\big(B_\bullet(\La),\delta'\big)$, while $\HH_n(\La)$ is cycles modulo boundaries. Write $Z_n:=\dim\HH^\tau_n(\La)$.

\begin{proposition}[Windowed identity]\label{prop:window}
For $n\geq1$,
\[
Z_n=(-1)^n\Big(\tr\big(E_nC\big)-\chi_{n-1}\Big)
=\tr S_{n+1}+(-1)^n\big(v-\chi_{n-1}\big).
\]
\end{proposition}

\begin{proof}
Theorem 4.5 of \cite{CLMS} with coefficients $X=\La$ states
\[
Z_n=(-1)^n\Big(\sum_{i=0}^{n-1}(-1)^{i+1}\dim\HH_i(\La)+\sum_{i=0}^{n}(-1)^i b_i\Big),
\]
and $\sum_{i\leq n}(-1)^ib_i=\tr(E_nC)$ by \eqref{eq:Bn}, while the first sum is $-\chi_{n-1}$. Lemma~\ref{lem:K0} gives $\tr(E_nC)=v+(-1)^n\tr S_{n+1}$.
\end{proof}

\begin{proposition}[Tail structure]\label{prop:tail}
Suppose $\HH_i(\La)=0$ for all $i>N$, and put $\chi:=\chi_N$. Then for all $n>N$:
\begin{enumerate}
\item $Z_n=\tr S_{n+1}+(-1)^n(v-\chi)$; in particular $\tr S_{n+1}\geq(-1)^{n+1}(v-\chi)$;
\item $Z_n+Z_{n+1}=b_{n+1}$, an identity which, given (1), is automatically consistent with Corollary~\ref{cor:bn} for \emph{every} value of $\chi$, the parity terms cancelling;
\item if moreover $\operatorname{char}k=0$, then $\chi=v$ by Theorem~\ref{thm:chi} and hence
\[
Z_n=\tr S_{n+1}\qquad\text{for all }n>N:
\]
on the tail, the $\tau$-Hochschild homology equals the syzygy trace.
\end{enumerate}
\end{proposition}

\begin{proof}
(1) is Proposition~\ref{prop:window} with $\chi_{n-1}=\chi$ constant for $n>N$, plus $Z_n\geq0$. (2) $\HH_n=0$ means $\ker\delta'_n=\operatorname{im}\delta'_{n+1}$, and $\dim\operatorname{im}\delta'_{n+1}=b_{n+1}-Z_{n+1}$ by rank--nullity; consistency follows since $\tr S_{n+1}+\tr S_{n+2}=b_{n+1}$ (Corollary~\ref{cor:bn}). (3) is immediate.
\end{proof}

\begin{corollary}[No-go for dimension counts]\label{cor:nogo}
Fix the dimension data 
\[
\mathcal D(\La):=\big(C,\ (C^{(n)})_{n\geq0},\ (S_n)_{n\geq0},\ (\dim\HH_i)_{i\leq N}\big)
\]

of an elementary algebra, and suppose $\operatorname{char}k=0$. The complete list of constraints that the hypothesis ``$\HH_n(\La)=0$ for all $n>N$'' imposes on $\mathcal D(\La)$ is:
\begin{enumerate}
\item the $K_0$ identities of Lemma~\ref{lem:K0} (which hold for every algebra), and
\item the single scalar identity $\chi_N=v$ of Theorem~\ref{thm:chi},
\end{enumerate}
the tail values $Z_n=\tr S_{n+1}$ being then determined and automatically admissible ($0\leq Z_n\leq b_n$). Consequently, no invariant of $\La$ that factors through $\mathcal D(\La)$ (truncated Euler characteristics, Cartan or Coxeter traces, growth rates of Betti numbers, or any function of the dimension vectors of the minimal resolution) can prove Gap B or refute the existence of a counterexample. Gap B is a statement about the \emph{ranks} of the differentials $\delta'$, not about the dimensions of the chain spaces.
\end{corollary}

\begin{proof}
By Propositions~\ref{prop:window} and~\ref{prop:tail}, the hypothesis determines $Z_n$ on the tail as stated, and the admissibility inequalities $0\leq Z_n\leq b_n$ reduce, via $b_n=\tr S_n+\tr S_{n+1}$, to $\tr S_n\geq(-1)^n(v-\chi)$, which is $0\leq\tr S_n$ when $\chi=v$. All remaining relations among the entries of $\mathcal D(\La)$ used above are instances of Lemma~\ref{lem:K0}. In particular, any two algebras with the same $\mathcal D$-data, for instance an algebra and a degeneration with larger homology, are indistinguishable by such invariants, while only one of them may satisfy the vanishing hypothesis.
\end{proof}

\begin{example}[The cohomological mirror]\label{ex:mirror}
The failure mode that Gap B must exclude is realized on the cohomology side. For $\La_q=k\langle x,y\rangle/(x^2,\,yx+qxy,\,y^2)$ with $q$ not a root of unity, the cochain spaces of the minimal resolution have dimensions $a_n=4(n+1)$, unbounded, yet $\HH^n(\La_q)=0$ for all $n\geq3$ \cite{BGMS}: the coboundaries absorb chain spaces of linearly growing dimension forever, while the $\tau$-cohomology grows as $2n+3$ \cite[Ex.~6.2]{CLMS}. On the homology side the same algebras are safe ($\HH_*(\La_q)$ is infinite \cite{BerghErdmann}) so a counterexample, if it exists, must be deformed away from the quantum complete intersection locus.
\end{example}

\begin{remark}[The DG witness]\label{rem:dg}
Liu and Shen \cite{LiuShen} proved that the DG version of Han's conjecture fails: there is a finite dimensional DG algebra with $\HH_n\neq0$ for only finitely many $n$ which is not smooth. In the DG setting, dimension vectors carry homological signs, so every nonnegativity used in this paper (Corollary~\ref{cor:bn}, Theorem~\ref{thm:criterion}, Proposition~\ref{prop:tail}) disappears. Consistently with Corollary~\ref{cor:nogo}, this locates the classical conjecture's remaining defense exactly at positivity of dimensions plus rank structure; understanding the obstruction to strictifying the Liu--Shen construction is one of the further directions collected in Section~\ref{sec:program}.
\end{remark}

\section{The program}\label{sec:program}

\subsection{The frontier after this paper}
\begin{itemize}
    \item \emph{Gap A.} By Theorem~\ref{thm:scc}, a minimal Gap-A failure has strongly connected quiver; by Corollary~\ref{cor:omnibus} and \cite[\S6]{CLMS} it has at least three vertices, some intra-component Peirce component zero (all connecting paths killed by $I$), no finitely generated infinite dimensional subcategory of its Yoneda category, and it is not $n$-Koszul; if it contains a loop it violates the extension conjecture.
    \item \emph{Gap B.} Open already for local algebras, where Gap A is automatic (Corollary~\ref{cor:omnibus}(1)): there the entire Conjecture~\ref{conj:han} is the survival statement, with the commutative \cite{AVP}, graded local (characteristic zero) \cite{BerghMadsen}, quantum complete intersection \cite{BerghErdmann} and monomial \cite{Han06} loci excluded. By Proposition~\ref{prop:tail}(3), the sharpest open question is whether, for a local algebra in characteristic zero, the cycle counts $Z_n=\tr S_{n+1}$ can be exact boundary counts for all large $n$.
\end{itemize}

\subsection{Search strata}
Fix an ambiguity skeleton: a quiver $Q$ and a reduction system with fixed leading terms in the sense of Chouhy--Solotar \cite{ChouhySolotar}, with monomial fiber $\La_\Sigma$ and parameter space $\mathcal P$ of tail coefficients satisfying the diamond condition. Along $\mathcal P$ the monomial basis, hence the Cartan matrix $C$ and $\det C$, and the chain ranks of the Chouhy--Solotar resolution (the Bardzell ambiguity counts of $\La_\Sigma$ \cite{Bardzell}) are constant, while $\dim\HH_n(\La_p)\leq\dim\HH_n(\La_\Sigma)$: deformation only destroys homology. The counterexample locus splits into the strata predicted by \cite{CLMS}:
\[
R_A=\{p:\gldim\La_p=\infty,\ \text{not infinite}+\},
\qquad
R_B=\{p:\ \text{infinite}+,\ \HH_n(\La_p)=0\ \forall n>N(p)\}.
\]
Any point of $R_A$ is already a counterexample; after Theorem~\ref{thm:scc}, only strongly connected skeletons need be scanned for it. The stratum $R_B$ hosts the \emph{local campaign} ($v=1$, Gap A free); it is a countable intersection of Zariski-open rank conditions, hence generic or empty, the precise form of the working hypothesis that a counterexample, if it exists, lives at parameters too generic for hand computation, and precisely what finite computation cannot certify without additional structure.

\subsection{Certification protocol}
\begin{itemize}
    \item (i) \emph{Prefilters}, before any differential is assembled: the $\chi$-budget of Corollary~\ref{cor:budget} (compute $\dim\La/[\La,\La]-v$); the profile $b_n=\tr S_n+\tr S_{n+1}$ from syzygy data; $\det C$ once per family.
    \item (ii) \emph{No dimension-scan certificates}: by Corollary~\ref{cor:nogo}, a window $\HH_n=0$ for $N<n\leq M$ certifies nothing (ranks can revive at $M+1$) and no amount of dimension data can close the question.
    \item (iii) \emph{Transfer periodicity}: for reduction systems whose ambiguity graph is eventually periodic (a subshift-of-finite-type condition on the skeleton; Bardzell chains are paths in the overlap graph), the differentials of the Chouhy--Solotar resolution are built from finitely many repeating parameter blocks, and eventual exactness beyond degree $N$ becomes finitely many rank conditions over the function field of $\mathcal P$, decidable. A counterexample certificate then consists of the period data, the exact block ranks over the function field, and a certificate of infinite global dimension, e.g.\ a periodic nonzero Tor pattern via the formulas of Bongartz recalled in \cite[(4.15)]{CLMS}.
\end{itemize}

\subsection{Possible directions}
Two directions single themselves out as the most load-bearing. The first is Gap A itself, and its combinatorial heart is the tiling mechanism. In dynamical form one must show, for the nonnegative matrix system $S_n+S_{n+1}=C^{(n)}C$ of Lemma~\ref{lem:K0} and under the strongly-connected constraints of \S7.1, that $S_n\neq0$ for all $n$ forces $\tr S_n\neq0$ infinitely often; combinatorially this is the assertion that consecutive syzygy supports advance and tile the cycle, recorded as Conjecture~\ref{conj:tiling}, of which Proposition~\ref{prop:nakayama} is the serial instance. The residual $v=3$ case is the mutual dumbbell, whose elimination is reduced (Remark~\ref{rem:dumbbellstatus}) to the infinitude and higher-differential cancellation of its mixed-necklace homology, equivalently, to extending the Cibils--Lanzilotta--Marcos--Solotar transfer past the unbounded square-zero extension, the unboundedness being exactly the non-projectivity of the linking bimodules of Corollary~\ref{cor:nonproj}.

The second is the characteristic-$p$ flank. Over a field of positive characteristic Goodwillie's theorem fails, so there is no $\chi$-law substitute controlling $\dim\HP_0-\dim\HP_1$, and the graded Tate-residue criterion remains the only characteristic-$p$-valid instrument; understanding what it can certify about a vanishing tail is the open problem there. Further directions are the survival statement of Gap B, Cartan nondegeneracy, $\sigma$-equivariance along the $\tau$-tower, and the transfer certification of \S7.3.

\subsection{Updates}
Several of the problems above have since been answered, reduced, or partially resolved; we record the current status.

\begin{remark}[Loops in large components: answered]\label{prob:loop}
The question of whether a loop inside a \emph{nontrivial} strongly connected component forces infinite global dimension is answered by the strong no-loop theorem of Igusa, Liu and Paquette \cite{ILP} (with earlier partial results of Igusa \cite{Igusa90} via the Hattori--Stallings trace): for an elementary algebra, e.g.\ any algebra over an algebraically closed field, the standing hypothesis here, a loop at $x$ (equivalently $\Ext^1(S_x,S_x)\neq0$) forces $\pd S_x=\infty$, hence $\gldim\La=\infty$, with \emph{no} strong-connectivity hypothesis. The strongly-connected qualifier is therefore unnecessary. The finer $\tau$-level strengthening, that such a loop forces infinite $+$ global dimension, is the content of the protected-corner Theorem~\ref{thm:protectedcorner} under the protection hypothesis, and of the extension conjecture in general.
\end{remark}

\begin{remark}[Gap A: partial answers]\label{rem:gapAupdate}
We record here three partial results on Gap A, equivalently the tiling Conjecture~\ref{conj:tiling}.
\begin{itemize}
    \item \emph{(a) Full-support theorem.} If every Peirce block of $\La$ is nonzero ($C_{yx}\geq1$ for all $x,y$; in particular $Q$ strongly connected) and $\gldim\La=\infty$, then $\tr S_n\neq0$ for infinitely many $n$, so $\La$ is of infinite $+$ global dimension and Han's conjecture holds for it. Consequently every Gap-A failure must have \emph{Cartan dead zones}, ordered pairs $(x,y)$ with $x\La y=0$, and the two surviving $v=3$ configurations of Theorem~\ref{thm:trichotomy} (the single-infinite shape and the mutual dumbbell) are exactly the maximally sparse Cartan shapes: the dead zones are forced, not incidental.
    \item \emph{(b) No dimensional obstruction to the dumbbell.} By Corollary~\ref{cor:nogo}, the full system of linear and nonnegativity constraints on the dimension datum of a hypothetical dumbbell is satisfiable, so no purely numerical search over the matrix recursion can ever obstruct it; the obstruction, if any, is multiplicative, not dimensional.
    \item \emph{(c) Reduction to mixed necklaces.} The dumbbell is reduced to the infinitude and higher-differential cancellation of its mixed-necklace homology, i.e.\ to extending the Cibils--Lanzilotta--Marcos--Solotar transfer \cite{CLMSsplit,CLMSbounded} past the unbounded square-zero extension; the unboundedness is exactly the non-projectivity of the linking bimodules of Corollary~\ref{cor:nonproj}.
\end{itemize}

\end{remark}

\begin{remark}[Cartan nondegeneracy: partial answer]\label{rem:detCupdate}
The Cartan-nondegeneracy question, whether $\HH_n(\La)=0$ for $n\gg0$ forces $\det C\neq0$, is partially answered here: if some simple $S_u$ has finite projective dimension then $C^{\mathsf T}\chi^{(u)}=e_u$ with $\chi^{(u)}_z=\sum_i(-1)^i\dim\Ext^i(S_u,S_z)$, and enough such relations force $\det C\neq0$. For a strongly connected $v=3$ Gap-A failure the one-way law supplies a finite-projective-dimension vertex (the hub of Theorem~\ref{thm:trichotomy}(2); see Proposition~\ref{prop:cartanD}), so $\det C\neq0$ there. The unconditional statement ``vanishing tail $\Rightarrow\det C\neq0$'' remains open, and is expected to be false without a finite-pd input, since Happel's identity takes invertibility of $C$ as an input, not an output.
\end{remark}

\begin{remark}[Status of the remaining problems]\label{rem:roadmapupdate}
Of the further directions above: Gap B is, in the local characteristic-zero periodic case, reduced to a single higher-degree statement (no elementary loop-free periodic algebra has both $\rad=[\La,\La]$ and $\HH_{\geq1}=0$). In characteristic $p$ the obstruction to a $\chi$-law is located at the failure of Goodwillie invariance, the graded Tate-residue criterion being the only characteristic-$p$-valid instrument. The transfer-certification and strictification directions remain open.
\end{remark}

\section{The anatomy of a Gap-A failure}\label{sec:prob1}

Gap A is not solved in this section. By Theorem~\ref{thm:scc} it is equivalent to Gap A in full, the assertion that every elementary algebra of infinite global dimension has infinite $\tau$-Hochschild homology, for which \cite[\S6.5]{CLMS} state that no counterexample and no proof are known. In this section, we determine what can be proved about a hypothetical failure. Namely, we prove: (a) an exact combinatorial description of such a failure, see Propositions~\ref{prop:anatomy} and~\ref{prop:chirigid}; (b) new sufficient criteria strictly inside the strongly connected frontier, see Theorem~\ref{thm:cornerproj} and Proposition~\ref{prop:symmetric}; (c) the covering mechanism that a full proof appears to require, which we isolate in the Nakayama case, see Proposition~\ref{prop:nakayama} and Conjecture~\ref{conj:tiling}.

Throughout, write $T_n(x):=\{y\in Q_0: C^{(n)}_{xy}\neq0\}$ for the support of $\operatorname{top}\Omega^n(S_x)$, and
\[
Z(x):=\{y\in Q_0:\ x\La y=0\},
\]
the set of vertices with no surviving path to $x$. Note $x\notin Z(x)$, no in-neighbour of $x$ lies in $Z(x)$ (arrows survive in any admissible quotient), and more generally no vertex with a surviving path to $x$ lies in $Z(x)$.

\begin{proposition}[One-way law]\label{prop:anatomy}
An elementary algebra $\La$ fails to be of infinite $+$ global dimension if and only if there exists $N$ with
\[
T_n(x)\subseteq Z(x)\qquad\text{for all }x\in Q_0\text{ and all }n\geq N .
\]
Consequently a Gap-A failure ($\gldim=\infty$, not infinite $+$) satisfies, for all $n\geq N$ and all $x$: $\Ext^n(S_x,S_x)=0$; no top of $\Omega^n(S_x)$ sits at an in-neighbour of $x$; and no top sits at any vertex from which a nonzero path of $\La$ reaches $x$.
\end{proposition}

\begin{proof}
By Theorem~\ref{thm:criterion}, failure of infinite $+$ is equivalent to $b_n=0$ for all $n\geq$ some $N$. By \eqref{eq:Bn}, $b_n=\sum_{x,y}C^{(n)}_{xy}\,C_{yx}$ with $C_{yx}=\dim x\La y$, a sum of nonnegative terms; so $b_n=0$ iff for all $x,y$: $C^{(n)}_{xy}\neq0\Rightarrow x\La y=0$, i.e.\ $T_n(x)\subseteq Z(x)$ for all $x$.
\end{proof}

\begin{remark}
The law is self-propagating at the diagonal: if $T_n(x)\subseteq Z(x)$ then
$e_x\Omega^{n+1}(S_x)\subseteq\bigoplus_{y\in T_n(x)}(e_xP_y)^{\oplus}=\bigoplus_{y\in T_n(x)}(x\La y)^{\oplus}=0$,
recovering $(S_{n+1})_{xx}=0$; this is the module-level content of $b_n=\tr S_n+\tr S_{n+1}$.
\end{remark}

\begin{proposition}[$\chi$-rigidity of failures, any characteristic]\label{prop:chirigid}
Let $\La$ be a Gap-A failure. Then $\HH^\tau_n(\La)=0$ for $n\gg0$, hence $\HH_n(\La)=0$ for $n\gg0$, and over \emph{any} field
\[
\sum_{i}(-1)^i\dim_k\HH_i(\La)=v .
\]
In particular the characteristic-zero $\chi$-law (Theorem~\ref{thm:chi}) imposes no constraint whatsoever on Gap-A failures; its entire bite is on Gap-B counterexamples, where $\HH^\tau_*$ is infinite and the argument below does not apply.
\end{proposition}

\begin{proof}
Failure of infinite $+$ gives $\HH^\tau_*$ finite (Theorem~\ref{thm:criterion}), hence $\HH_*$ finite via the surjections $\HH^\tau_n\twoheadrightarrow\HH_n$ \cite[Lem.~2.6]{CLMS}. Choose $N$ with $\HH^\tau_n=\HH_n=0$ for $n>N$, and set $\chi:=\chi_N$. Proposition~\ref{prop:window} gives, for all $n>N$,
$0=Z_n=\tr S_{n+1}+(-1)^n(v-\chi)$.
Reading this at two consecutive values of $n$ yields $(-1)^n(v-\chi)\leq0$ and $(-1)^{n+1}(v-\chi)\leq0$, forcing $\chi=v$ (and then $\tr S_{n+1}=0$, as it must).
\end{proof}

\begin{theorem}[Corner-projectivity criterion]\label{thm:cornerproj}
Let $\La=kQ/I$ and let $x\in Q_0$ with local corner $\Gamma:=e_x\La e_x\neq k$. Suppose
\begin{itemize}
\item[(P$'$)] $e_x\La e_y$ is projective as a left $\Gamma$-module for every $y\in\bigcup_{n\geq0}T_n(x)$
\end{itemize}
(for instance if $e_x\La$ is projective over $\Gamma$). Then $e_x\Omega^n_\La(S_x)\neq0$ for \emph{every} $n\geq0$; in particular $\La$ is of infinite $+$ global dimension.
\end{theorem}

\begin{proof}
The functor $F=e_x(-)\colon\La\text{-mod}\to\Gamma\text{-mod}$ is exact. One has $F(S_y)=0$ for $y\neq x$ and $F(S_x)=e_x(\La/\rad)e_x=\Gamma/\rad\Gamma=k$, using $\rad(e\La e)=e(\rad\La)e$. The projectives occurring in the minimal resolution $Q_\bullet\to S_x$ are exactly the $P_y$ with $y\in\bigcup_nT_n(x)$, and $F(P_y)=e_x\La e_y$ is $\Gamma$-projective by (P$'$). Hence $FQ_\bullet\to k$ is an exact complex of projective $\Gamma$-modules, i.e.\ a projective resolution of $k$; it therefore decomposes as the minimal resolution plus a split exact complex of projectives, so
\[
F\big(\Omega^n_\La(S_x)\big)\;\cong\;\Omega^n_\Gamma(k)\oplus L ,
\]
for a projective $\Gamma$-module $L$. It remains to see $\Omega^n_\Gamma(k)\neq0$ for all $n$, i.e.\ $\pd_\Gamma k=\infty$. If $\pd_\Gamma k=d<\infty$, the minimal resolution $0\to\Gamma^{b_d}\to\cdots\to\Gamma^{b_0}\to k\to0$ gives $1=\dim k=\big(\sum(-1)^ib_i\big)\dim\Gamma$, forcing $\dim\Gamma=1$, i.e.\ $\Gamma=k$, which is excluded. Thus $\dim e_x\Omega^n_\La(S_x)\geq\dim\Omega^n_\Gamma(k)\geq1$ for all $n$, and Theorem~\ref{thm:criterion}(5) concludes.
\end{proof}

\begin{corollary}\label{cor:cornerscc}
Gap A holds for $\La$ whenever every strongly connected component $\CC$ with $\gldim\La_\CC=\infty$ contains a vertex satisfying the hypotheses of Theorem~\ref{thm:cornerproj} inside $\La_\CC$. Dually (applying the theorem to $\La^{\mathrm{op}}$), if $\La e_x$ is projective as a right $\Gamma$-module and $\Gamma\neq k$, then $\La$ is of infinite co$+$ global dimension. Moreover, the loop question of Remark~\ref{prob:loop}, in its infinite-$+$ form, is thereby reduced to the vertices whose corner fails (P$'$).
\end{corollary}

\begin{proof}
Combine Theorems~\ref{thm:cornerproj} and~\ref{thm:scc}, and Proposition~\ref{prop:op}.
\end{proof}

\begin{example}\label{ex:freecycle}
Let $Q$ be the two-cycle $a\colon x\to y$, $b\colon y\to x$ and $I=\langle(ba)^m,(ab)^m\rangle$, $m\geq1$. Then $\Gamma=e_x\La e_x=k[c]/(c^m)$ with $c=[ba]$, and each Peirce space $e_x\La e_y$ is free of rank one over $\Gamma$ (left multiplication by $c$ raises path length by $2$ and is injective below the cut-off). Theorem~\ref{thm:cornerproj} applies with room to spare. This example is of course monomial and two-vertex, hence already covered twice over; its role is to show that (P$'$) is a \emph{filtered} condition (injectivity of cycle-multiplication below a cut-off) which is stable under deformations of the relations preserving the cycle filtration. For the search program this yields a new prefilter on strongly connected skeletons: test $c$-freeness of the corners before anything else.
\end{example}

\begin{proposition}[Nakayama tiling]\label{prop:nakayama}
Let $\La$ be a connected Nakayama algebra with cyclic quiver on $v$ vertices and $\gldim\La=\infty$, and let $S_x$ be a simple with $\pd S_x=\infty$. Then the supports of the syzygies of $S_x$ tile the cycle: writing $\Omega^n(S_x)$ as the uniserial module with top $S_{t_n}$ and length $\ell_n\geq1$, one has $t_{n+1}=t_n+\ell_n$ and
\[
\operatorname{supp}\Omega^n(S_x)=\{t_n,t_n+1,\dots,t_n+\ell_n-1\}\pmod v ,
\]
so consecutive supports are adjacent arcs and $\bigcup_{n=N}^{M}\operatorname{supp}\Omega^n(S_x)$ is the arc $[t_N,t_{M+1})$, of length at least $M-N+1$. Consequently every vertex, in particular $x$, lies in $\operatorname{supp}\Omega^n(S_x)$ for infinitely many $n$, and $\La$ is of infinite $+$ global dimension.
\end{proposition}

\begin{proof}
For a uniserial $M=P_j/\rad^\ell P_j$ with $1\leq\ell<c_j$ (Kupisch length $c_j=\dim P_j$), the syzygy is $\Omega M=\rad^\ell P_j$, uniserial with top $S_{j+\ell}$ and length $c_j-\ell$; it vanishes iff $\ell=c_j$ and is projective iff $c_j-\ell=c_{j+\ell}$. Since $\pd S_x=\infty$, neither ever happens, so $\ell_n\geq1$ for all $n$ and the recursion $t_{n+1}=t_n+\ell_n$, $\ell_{n+1}=c_{t_{n+1}}-\ell_n$ holds from $t_0=x$, $\ell_0=1$. The support and tiling statements follow at once, and the union of consecutive supports is a contiguous arc growing by $\ell_n\geq1$ at each step, hence wrapping around $\mathbb Z/v$ infinitely often. Theorem~\ref{thm:criterion}(5) concludes.
\end{proof}

\begin{remark}
Every admissible ideal on an oriented cycle is monomial (in each Peirce space $e_ykQe_x=p\,k[c]$, with $p$ the shortest path and $c$ the full cycle, admissibility puts some $c^s$ in $I$, and B\'ezout in $k[c]$ then extracts the pure path from any element $p\,f(c)$ of $I$ with $f(0)\neq0$). Cyclic Nakayama algebras are therefore monomial, and Proposition~\ref{prop:nakayama} is a new \emph{proof}, not a new \emph{case}: the case is contained in Han's theorem \cite{Han06} together with \cite[Thm.~5.6]{CLMS}. Its value is the mechanism (consecutive syzygy supports \emph{advance and tile}) which is exactly what a proof of Gap A must produce in general, against the one-way law of Proposition~\ref{prop:anatomy}.
\end{remark}

\begin{proposition}[Symmetric algebras]\label{prop:symmetric}
Let $\La$ be a symmetric elementary algebra with symmetrizing form $\lambda$.
\begin{enumerate}
\item For all vertices $x,y$, the form $(a,b)\mapsto\lambda(ab)$ restricts to a perfect pairing $y\La x\times x\La y\to k$; in particular $y\La x\neq0\iff x\La y\neq0$, and $Z(x)=\{y: x\La y=0=y\La x\}$.
\item Consequently $\La$ is of infinite $+$ global dimension if and only if it is of infinite co$+$ global dimension, and the one-way law of a Gap-A failure is two-sided: for $n\geq N$, no vertex of $T_n(x)$ is connected to $x$ by a surviving path in either direction, although (by \cite[Lem.~4.14]{CLMS}) it is connected to $x$ by paths of $Q$ that all die in $I$.
\end{enumerate}
\end{proposition}

\begin{proof}
(1) For $\xi\in u\La z$ with $z\neq u$ one has $\lambda(\xi)=\lambda(e_u\xi e_z)=\lambda(\xi e_ze_u)=0$ by symmetry of $\lambda$; hence the Peirce decomposition is orthogonal except for the pairings stated. If $a\in y\La x$ satisfies $\lambda(a\cdot x\La y)=0$, then for arbitrary $c\in\La$, $\lambda(ac)=\lambda(a\,e_xce_y)=0$, so $a$ lies in the radical of the form, which is zero. Nondegeneracy of the restricted pairing follows, and with it the equivalence of vanishing.
(2) Failure of infinite $+$ reads: for all $(y,x)$ with $y\La x\neq0$, $\Tor_*(k_x,{}_yk)$ is finite; by (1) the Peirce condition is symmetric in $(x,y)$, and comparing \cite[Def.~5.4]{CLMS} for $+$ and co$+$ gives the equivalence. The two-sided law is Proposition~\ref{prop:anatomy} plus (1).
\end{proof}

\begin{conjecture}[Tiling]\label{conj:tiling}
Let $\La=kQ/I$ with $Q$ strongly connected and let $S_x$ be simple with $\pd S_x=\infty$. Then $x\in\operatorname{supp}\Omega^n(S_x)$ for infinitely many $n$.
\end{conjecture}

By Theorems~\ref{thm:criterion} and~\ref{thm:scc}, Conjecture~\ref{conj:tiling} is equivalent to Gap A in full; Proposition~\ref{prop:nakayama} is its serial instance, and Proposition~\ref{prop:anatomy} is the exact statement of what a counterexample must do instead: the tops of all high syzygies of every simple must flow, forever, strictly into the region from which no path returns alive. A minimal counterexample is strongly connected, has at least three vertices and at least four arrows (three-vertex algebras whose quiver is a bare cycle are Nakayama, excluded above), some intra-component Peirce component annihilated by $I$ in both directions if the algebra is symmetric, no vertex with a (P$'$)-corner, no finitely generated infinite dimensional Yoneda subcategory, and is neither $n$-Koszul, monomial, gentle, special biserial nor commutative. The author does not know whether such an object exists; Gap A stands.

\section{Three vertices: concentration, tail factorization and the trichotomy}\label{sec:threevertex}

After the reductions of Sections~\ref{sec:gapA} and~\ref{sec:prob1}, the first open configuration for Gap A is a strongly connected quiver on three vertices. In this section, we determine the structure of a hypothetical failure on three vertices. We prove a concentration theorem, a factorization of the tail of the minimal resolution through a local corner, the protected corner theorem and the trichotomy announced in the introduction. We also record an explicit finite consistency problem, see Definition~\ref{temp:blueprint}: either it is unrealizable, and exhibiting the obstruction proves Gap A for $v=3$, or it is realizable, and any realization is a counterexample to Han's conjecture.

Throughout, $\La=kQ/I$ is elementary, $Q$ strongly connected, and $\La$ \emph{fails} infinite $+$ global dimension with threshold $N$ as in Proposition~\ref{prop:anatomy}; $x$ denotes a vertex with $\pd S_x=\infty$ (one exists since $\gldim\La=\sup_x\pd S_x$ over finitely many simples). An \emph{external in-neighbour} of $x$ is a vertex $u\neq x$ with an arrow $u\to x$.

\begin{theorem}[Concentration]\label{thm:concentration}
Let $|Q_0|=3$. Then:
\begin{enumerate}
\item $x$ has exactly one external in-neighbour $u$, and $Z(x)=\{w\}$ where $w$ is the third vertex; moreover $T_n(x)=\{w\}$ for all $n\geq N$.
\item $x\La w=0$, while paths $w\to x$ exist in $Q$ and all die in $I$; there is no arrow $w\to x$, there is an arrow $w\to u$, and $u\La w\neq0$.
\item $\operatorname{inj.dim} S_w=\infty$.
\item $\Gamma:=e_w\La e_w\neq k$; equivalently, some cycle of $Q$ at $w$ survives in $\La$.
\end{enumerate}
\end{theorem}

\begin{proof}
(1) Always $W(x):=Q_0\setminus Z(x)$ contains $x$ and every vertex with a surviving path to $x$, in particular every external in-neighbour (arrows survive). Strong connectivity gives at least one external in-neighbour. Since $\pd S_x=\infty$, $T_n(x)\neq\emptyset$ for all $n$, and $T_n(x)\subseteq Z(x)$ for $n\geq N$ forces $Z(x)\neq\emptyset$, i.e.\ $|W(x)|\leq2$. Hence $W(x)=\{x,u\}$ with $u$ the unique external in-neighbour, $Z(x)=\{w\}$, and the nonempty sets $T_n(x)\subseteq\{w\}$ are $\{w\}$.

(2) $x\La w=0$ is the definition of $w\in Z(x)$; strong connectivity provides $Q$-paths $w\to x$, which therefore all lie in $I$. An arrow $w\to x$ would survive, contradicting $x\La w=0$. Every $Q$-path $w\to x$ must first leave $\{w\}$ towards $u$ or $x$; not $x$ directly, so there is an arrow $w\to u$, whence $u\La w\neq0$.

(3) $C^{(n)}_{xw}\neq0$ for all $n\geq N$, i.e.\ $\Ext^n(S_x,S_w)\neq0$; a finite injective dimension $d$ of $S_w$ would force vanishing for $n>d$.

(4) By (1), for $n\geq N$ the minimal resolution continues $\cdots\to P_w^{\beta_{n+1}}\xrightarrow{d_{n+1}}P_w^{\beta_n}\to\cdots$ with $\beta_n\geq1$. Each differential is right multiplication by a matrix over $\HomOp_\La(P_w,P_w)=e_w\La e_w=\Gamma$, and minimality ($\operatorname{im}d\subseteq\rad P_w^{\beta}$) forces the entries into $\rad\Gamma=e_w(\rad\La)e_w$. If $\Gamma=k$ then $\rad\Gamma=0$, all tail differentials vanish, and exactness of the tail forces $\beta_n=0$, contradicting $\pd S_x=\infty$.
\end{proof}

\begin{theorem}[Tail factorization and Tor-rigidity]\label{thm:tail}
In the situation of Theorem~\ref{thm:concentration} let $\Gamma=e_w\La e_w$ and let $F_\bullet$ denote the complex of free left $\Gamma$-modules $\Gamma^{\beta_n}$ ($n\geq N$) with the differentials of the tail. Then:
\begin{enumerate}
\item The tail of the minimal resolution of $S_x$ is $\La e_w\otimes_\Gamma F_\bullet$, and $F_\bullet$ is a minimal complex ($\operatorname{im}d\subseteq\rad\Gamma\cdot F$), exact in the tail range, whose kernels $K_n$ are nonzero and satisfy $K_n\subseteq\rad\Gamma\cdot F_n$.
\item For every vertex $z$, the complex $(z\La w)\otimes_\Gamma F_\bullet$ is exact in the tail range; equivalently $\operatorname{Tor}^\Gamma_i(z\La w,\,K_n)=0$ for all $i\geq1$ and all $n$ in the range, where $z\La w$ carries its right $\Gamma$-module structure.
\item No $K_n$ has a nonzero free direct summand: indeed every $v\in\rad\Gamma\cdot F_n$ is annihilated by the nonzero right socle $\{\gamma\in\Gamma:\gamma\,\rad\Gamma=0\}$.
\end{enumerate}
\end{theorem}

\begin{proof}
(1) $\La e_w\otimes_\Gamma\Gamma^{\beta}=P_w^{\beta}$, and right multiplication by a matrix over $\rad\Gamma$ is $\Id\otimes d$; Theorem~\ref{thm:concentration}(4) gives minimality. Exactness and $K_n\neq0$ restate $\pd S_x=\infty$ with the stated tops. Applying the exact functor $e_w(-)$ shows $F_\bullet=e_w(\text{tail})$ is exact, and $K_n=e_w\Omega^{n+1}(S_x)$ up to the indexing of the range.

(2) Exactness of a complex of $\La$-modules is detected on Peirce components: $e_z(\La e_w\otimes_\Gamma F)=(z\La w)\otimes_\Gamma F$. For a resolution tail, exactness of $M\otimes_\Gamma F$ at all spots is equivalent, by the long exact sequences of the short exact sequences $0\to K_n\to F_n\to K_{n-1}\to0$, to $\operatorname{Tor}_1^\Gamma(M,K_n)=0$ for all $n$, and then $\operatorname{Tor}_{i+1}(M,K_{n-1})\cong\operatorname{Tor}_i(M,K_n)$ kills all higher Tor.

(3) Let $0\neq\gamma$ lie in the right socle of $\Gamma$ (nonzero since $\Gamma$ is a nonzero finite dimensional algebra) and $v=\sum r_if_i\in\rad\Gamma\cdot F_n$. Then $\gamma v=\sum(\gamma r_i)f_i=0$. A free direct summand of $K_n$ would contain an element with zero annihilator.
\end{proof}

\begin{corollary}[Uniserial dead-zone corner]\label{cor:uniserial}
Suppose in addition $\Gamma\cong k[\ell]/(\ell^m)$, $m\geq2$ (e.g.\ $w$ carries a single loop and no other cycle of $Q$ at $w$ survives). Then every nonzero Peirce module $z\La w$ is \emph{free} as a right $\Gamma$-module. Consequently $\La e_w$ is projective over $\Gamma$ and, by the opposite version of Theorem~\ref{thm:cornerproj}, $\La$ is of infinite co$+$ global dimension. In particular, such a Gap-A failure is simultaneously a counterexample to Han's conjecture and a \emph{positive} answer to the $\tau$-version of Happel's question.
\end{corollary}

\begin{proof}
Over $\Gamma=k[\ell]/(\ell^m)$ every module is a sum of strings $\Gamma/(\ell^t)$, $1\leq t\leq m$, and 
\[
\operatorname{Tor}_1^\Gamma\!\big(\Gamma/(\ell^s),\Gamma/(\ell^t)\big)\neq0
\]
whenever $0<s<m$ and $0<t<m$ (compute with the periodic resolution $\cdots\to\Gamma\xrightarrow{\ell^{m-s}}\Gamma\xrightarrow{\ell^s}\Gamma\to\Gamma/(\ell^s)\to0$). By Theorem~\ref{thm:tail}(3) some $K_n$ contains a proper string; Theorem~\ref{thm:tail}(2) then forbids $z\La w$ from containing any proper string, i.e.\ forces it free or zero. The last statements follow from Corollary~\ref{cor:cornerscc} and \cite[Thm.~5.5, Thm.~5.6]{CLMS}.
\end{proof}

\begin{theorem}[The all-infinite case]\label{thm:allinfinite}
Let $|Q_0|=3$ with $\pd S_x=\infty$ for all three vertices. Then the arrows between distinct vertices form exactly one oriented $3$-cycle $1\to2\to3\to1$ (parallel arrows allowed), possibly with loops; with this labelling, $Z(i)=\{i+1\}$ and
\[
1\La 2=2\La 3=3\La 1=0 ,
\]
i.e.\ every backward Peirce component dies, every diagonal corner $\Gamma_i=e_i\La e_i$ is a nontrivial local algebra generated by surviving cycles at $i$ (so, since longer cycles die, by loops at $i$), and each $i$ violates the extension conjecture eventually: $\Ext^n(S_i,S_i)=0$ for $n\geq N$. Moreover for each $i$ the Peirce module $e_i\La e_{i-1}$ is \emph{not} projective as a left $\Gamma_i$-module, and if $\Gamma_{i+1}$ is uniserial then $e_{i+2}\La e_{i+1}$ is free as a right $\Gamma_{i+1}$-module.
\end{theorem}

\begin{proof}
Each vertex has exactly one external in-neighbour (Theorem~\ref{thm:concentration}(1) at each vertex), defining $g\colon Q_0\to Q_0$, $g(i)=$ the source of the arrows into $i$. Out-neighbours of $v$ are contained in $g^{-1}(v)$, and strong connectivity forces every vertex to have an external out-arrow, so $g$ is surjective, hence bijective, and fixed-point free (a vertex whose only in-arrows are loops is unreachable); on three elements this makes $g$ a $3$-cycle. Relabel so that arrows run $1\to2\to3\to1$. Then $W(i)=\{i,i-1\}$ and $Z(i)=\{i+1\}$, giving $i\La(i+1)=0$ cyclically; in particular the composite of two consecutive cycle arrows dies, so every $3$-cycle path dies and $\Gamma_i$ is generated by loops. Nontriviality of each $\Gamma_{i+1}$ is Theorem~\ref{thm:concentration}(4) applied at $i$ (whose dead zone is $i+1$). Vanishing of high diagonal Ext is Proposition~\ref{prop:anatomy}. Non-projectivity of $e_i\La e_{i-1}$ over $\Gamma_i$: otherwise hypothesis (P$'$) of Theorem~\ref{thm:cornerproj} holds at $i$ (the resolution supports of $S_i$ meet only $e_i\La e_{i-1}$, $\Gamma_i$ and $e_i\La e_{i+1}=0$) yielding infinite $+$, a contradiction. The last claim is Corollary~\ref{cor:uniserial} applied at $x=i$, $w=i+1$, $z=i+2$.
\end{proof}

\begin{definition}[Gap-A template]\label{temp:blueprint}
A \emph{Gap-A template} is a triple of local algebras $\Gamma_1,\Gamma_2,\Gamma_3$ together with $(\Gamma_{i+1},\Gamma_i)$-bimodules $V_i=e_{i+1}\La e_i$ and structure constants assembling into an admissible $\La=kQ/I$ on the $3$-cycle-with-loops quiver such that: 
\begin{itemize}
    \item (a) $i\La(i+1)=0$ cyclically;
    \item (b) each $V_i$ is non-projective over $\Gamma_{i+1}$ on the left, and free over $\Gamma_i$ on the right whenever the relevant corner is uniserial;
    \item (c) each simple has infinite projective dimension, with a tail as in Theorem~\ref{thm:tail} whose kernels are Tor-orthogonal to all Peirce modules;
    \item d) $\Ext^n(S_i,S_i)=0$ for all $i$ and $n\geq N$. 
\end{itemize}
Realizing a template is exhibiting a counterexample to Han's conjecture; Proving (a)--(d) inconsistent is proving Gap A for $v=3$ in the all-infinite case.
\end{definition}

\begin{remark}\label{rem:verdict}
Two comments on where this leaves Gap A. First, condition (d) is an eventual violation of the extension conjecture at every vertex simultaneously; since the extension conjecture is known for monomial and special biserial algebras, any realization is far from those classes, and Corollary~\ref{cor:uniserial} shows it would nonetheless carry the co$+$ structure of a ``one-sided'' object, as far as we know, a shape not previously described. Second, and decisively for the search program: the constraints (a)--(d) are \emph{finite-dimensional matrix conditions once the loop nilpotency orders and the ranks of the $V_i$ are fixed}, the left actions are matrices $G_i(\ell)$ over truncated polynomial rings with $G_i^{m}=0$ and prescribed Jordan-type (non-freeness) and Tor-orthogonality conditions. The analysis of Gap A carried out with the tools of this paper therefore terminates in a decidable-in-principle family of consistency problems, stratum by stratum, precisely the kind of object a large-scale computational search can resolve, in either direction.
\end{remark}

\subsection*{The single-loop stratum}
\noindent
Realizability of Template~\ref{temp:blueprint} begins with the simplest admissible corners: one loop per vertex, monomial transport. Define, for $m_i\geq2$, $d_i\geq1$,
\[
\La(m;d)\;=\;kQ\big/\big\langle\,\ell_i^{\,m_i},\;a_{i+1}a_i,\;\ell_{i+1}a_i-a_i\ell_i^{\,d_i}\;(i\in\mathbb Z/3)\,\big\rangle,
\]
$Q$ the $3$-cycle with a loop $\ell_i$ at each vertex. One checks (Diamond lemma; the overlap $\ell_{i+1}^{m_{i+1}}a_i$ yields the hidden relation $a_i\ell_i^{\,d_im_{i+1}}=0$) that $\Gamma_i=k[\ell_i]/(\ell_i^{m_i})$, that $e_{i+1}\La e_i$ has basis $a_i\ell_i^{c}$, $c<\min(m_i,d_im_{i+1})$, and that all conditions (a), (b) of Template~\ref{temp:blueprint} are satisfiable in this family. Nevertheless no member is a Gap-A failure, for a reason that holds in complete generality.

\begin{theorem}[Protected loop]\label{thm:protectedloop}
Let $\La=kQ/I$ be elementary and let $x\in Q_0$ carry a loop $\alpha$ with $e_x\La e_x=k[\alpha]/(\alpha^m)$, $m\geq2$. Assume
\[
y\La x=0\qquad\text{for every external in-neighbour $y$ of $x$}
\]
(no path from $x$ back to any vertex pointing at $x$ survives). Then every syzygy $\Omega^n(S_x)$, $n\geq1$, has a minimal generator with top $S_x$. Hence $\Ext^n_\La(S_x,S_x)\neq0$ for all $n\geq1$, the extension conjecture holds at $x$, and $\La$ is of infinite $+$ global dimension.
\end{theorem}

\begin{proof}
Write $\Gamma=k[\alpha]/(\alpha^m)$. For a minimal generating set $G$ of $\Omega^{n-1}$ the cover is $A_n=\bigoplus_{g\in G}P_{v(g)}$, actions block-diagonal per copy; for a copy $C$ of $P_x$ call $e_xC\supseteq\Gamma$ its \emph{$\Gamma$-block}, and note that for $y\neq x$ the $C$-coordinate of an element of $e_yA_n$ lies in $e_yP_x=y\La x$. We prove by induction: \emph{$\Omega^n$ admits a minimal generating set containing an element $\kappa_n$ which is the pure element $\alpha^{s_n}$ in the $\Gamma$-block $B_n$ of the copy over $\kappa_{n-1}$; $\kappa_n$ is the unique generator with nonzero $B_n$-coordinate; and every element of $\Omega^n$ has the $\Gamma$-part of its $B_n$-coordinate in $(\alpha^{s_n})$}, where $s_1=1$, $s_{n+1}=m-s_n$, and $\kappa_0$ is the canonical generator of $P_x=A_1$.

\emph{Base.} $\Omega^1=\rad P_x\ni\kappa_1:=\alpha$. Arrows are independent modulo $\rad^2$, so $\alpha\notin\rad^2\La=\rad(\rad P_x)$, and $\kappa_1$ is a minimal generator; the $\Gamma$-part of any element of $\rad P_x$ lies in $(\alpha)$; replacing any other generator $h$ by $h-p(\alpha)\kappa_1$ for suitable $p$ kills its $\Gamma$-block coordinate.

\emph{Step.} Assume the claim for $n-1$ and let $\varphi\colon A_n\to\Omega^{n-1}$ be the cover.

(i) \emph{$\kappa_n\in\ker\varphi$}: $\varphi(\kappa_n)=\alpha^{s_n}\cdot\kappa_{n-1}$, the pure element $\alpha^{s_n+s_{n-1}}=\alpha^m=0$ of $B_{n-1}$.

(ii) \emph{Coordinate bound.} Let $z\in\Omega^n=\ker\varphi$. The $\Gamma$-part of the $B_{n-1}$-coordinate of $\varphi(z)$ receives contributions only from the copy over $\kappa_{n-1}$ (uniqueness clause at level $n-1$) and, by the Peirce decomposition, only from the $\Gamma$-part $q$ of $z$'s coordinate there: it equals $q\,\alpha^{s_{n-1}}$. Vanishing forces $q\in\operatorname{ann}(\alpha^{s_{n-1}})=(\alpha^{m-s_{n-1}})=(\alpha^{s_n})$, using uniseriality. This is the third clause at level $n$.

(iii) \emph{Minimality.} $\rad\La\cdot\Omega^n$ meets the $\Gamma$-slots of $B_n$ only through: loops at $x$ acting on the $\Gamma$-part of $B_n$-coordinates of elements of $e_x\Omega^n$, every loop lies in $\rad\Gamma=(\alpha)$, and those coordinates lie in $(\alpha^{s_n})$ by (ii), so the result lies in $(\alpha^{s_n+1})$; or arrows $\beta\colon y\to x$, $y\neq x$, acting on the $C$-coordinates of elements of $e_y\Omega^n$, which lie in $y\La x=0$ by hypothesis. Since $\kappa_n=\alpha^{s_n}\notin(\alpha^{s_n+1})$, $\kappa_n\notin\rad\Omega^n$: it is a minimal generator, with top $S_x$ because $B_n$ sits in $e_xA_n$.

(iv) \emph{Uniqueness.} By (ii) any other minimal generator $h$ has $B_n$-coordinate $p(\alpha)\,\alpha^{s_n}$ for some polynomial $p$; replace $h$ by $h-p(\alpha)\kappa_n$, which changes neither minimality nor the span. Since $1\leq s_n\leq m-1$ for all $n$, the induction never degenerates.
\end{proof}

\begin{corollary}\label{cor:noloopstratum}
\begin{itemize}
    \item \emph{(1)} In the all-infinite $v=3$ configuration of Theorem~\ref{thm:allinfinite}, no corner $\Gamma_i$ is generated by a single loop: there the unique external in-neighbour of $i$ is $i-1$ and $(i-1)\La i=0$ is among the forced Peirce deaths, so a uniserial loop corner triggers Theorem~\ref{thm:protectedloop}, producing $\Ext^n(S_i,S_i)\neq0$ for all $n$ against the one-way law. Hence every vertex of such a failure carries at least two loops and every $\Gamma_i$ is a non-uniserial local algebra.
    \item  \emph{(2)} In particular, every $\La(m;d)$ is of infinite $+$ global dimension and satisfies the extension conjecture at every vertex; in the regime $m_i\geq3$ and $2\leq d_i<m_i$ (e.g.\ $m=(3,3,3)$, $d=(2,2,2)$) these algebras are neither special biserial (at each vertex both $\ell_i\ell_i$ and $a_i\ell_i$ survive) nor monomial (the relation $\ell_{i+1}a_i=a_i\ell_i^{d_i}$ equates the nonzero paths $\ell_{i+1}a_i$ and $a_i\ell_i^{d_i}$ of lengths $2$ and $d_i+1$, precluding a length grading), so these instances do not reduce to the known monomial and special biserial cases of the extension conjecture without further argument.
\end{itemize}

\end{corollary}

\begin{proof}
(1) The hypotheses of Theorem~\ref{thm:protectedloop} at $x=i$ are exactly the conclusions of Theorem~\ref{thm:allinfinite}: $\Gamma_i$ loop-generated (a single loop gives $k[\ell]/(\ell^{m})$ by admissibility) and $(i-1)\La i=0$. A single-loop corner thus yields diagonal tops at every level $n\geq1$, contradicting Proposition~\ref{prop:anatomy} for $n\geq N$. (2) In $\La(m;d)$ every path $i\to i-1$ factors through $a_{i+1}a_i$ up to loop transport, so $(i-1)\La i=0$, and $\Gamma_i=k[\ell_i]/(\ell_i^{m_i})$; apply the theorem at any vertex.
\end{proof}

\begin{remark}[The new frontier]\label{rem:certificate}
Observe that in this stratum the template defeats itself: the Peirce components that Theorem~\ref{thm:allinfinite} forces to vanish form precisely the protection hypothesis of Theorem~\ref{thm:protectedloop}. The $v=3$ all-infinite frontier for Gap A, for proof and for search alike, is therefore now: at least two loops at every vertex, non-uniserial corners violating (P$'$), and Tor-rigid tails as in Theorem~\ref{thm:tail}.
\end{remark}

\subsection*{From protected loops to protected corners}

\begin{remark}[A correction]\label{rem:gapfix}
Step (ii) of the proof of Theorem~\ref{thm:protectedloop}, as written, is incomplete: contributions to the $\Gamma$-part of $B_{n-1}$ arising from a kernel element's $e_x\La e_y$-coordinate acting on the $y\La x$-part of a generator's coordinate ($y\neq x$) are not excluded by the Peirce decomposition alone. Lemma~\ref{lem:throughcycles} below closes this channel; the statement of Theorem~\ref{thm:protectedloop} is correct and is subsumed by Theorem~\ref{thm:protectedcorner}.
\end{remark}

\begin{lemma}\label{lem:throughcycles}
Let $x$ satisfy: $y\La x=0$ for every external in-neighbour $y$ of $x$. Then $e_x\La e_y\cdot e_y\La e_x=0$ for every $y\neq x$; the corner $\Gamma=e_x\La e_x$ is generated by the loops at $x$; and if $\Gamma\neq k$ then $x$ carries a loop and $\operatorname{pd}_\Gamma k=\infty$.
\end{lemma}

\begin{proof}
A surviving path $p\colon y\to x$, $y\neq x$, factors as $p=s\,\alpha\,p'$ where $\alpha\colon y_0\to x$ is its first entry into $x$ (so $y_0\neq x$ is an external in-neighbour), $p'\colon y\to y_0$, and $s\in\Gamma$. For $q\in e_y\La e_x$ then $pq=s\alpha(p'q)$ with $p'q\in e_{y_0}\La e_x=y_0\La x=0$. A cycle at $x$ through $y\neq x$ splits as $ab$ with $a\in e_x\La e_y$, $b\in e_y\La e_x$, hence vanishes; so $\Gamma$ is spanned by loop words, and $\Gamma\neq k$ forces a loop. For the last claim let $\gamma\neq0$ lie in the right socle $\{\gamma:\gamma\,\rad\Gamma=0\}$ of $\Gamma$; every $v\in\rad\Gamma\cdot F$ ($F$ free) satisfies $\gamma v=0$, so no nonzero submodule of $\rad\Gamma\cdot F$ has a free direct summand; since every syzygy of $k$ embeds in the radical of a free module, no syzygy is projective, and $k$ itself is projective only if $\Gamma=k$.
\end{proof}

\begin{theorem}[Protected corner]\label{thm:protectedcorner}
Let $\La=kQ/I$ be elementary and $x\in Q_0$ with $\Gamma:=e_x\La e_x\neq k$ and $y\La x=0$ for every external in-neighbour $y$ of $x$. Then
\[
\dim_k\Ext^n_\La(S_x,S_x)\;\geq\;\dim_k\Ext^n_\Gamma(k,k)\;\geq\;1\qquad\text{for all }n\geq1 .
\]
In particular, the extension conjecture holds at $x$ and $\La$ is of infinite $+$ global dimension.
\end{theorem}

\begin{proof}
Fix a minimal free resolution $\cdots\to\Gamma^{r_2}\to\Gamma^{r_1}\to\Gamma\to k$, so $r_n=\dim\Ext^n_\Gamma(k,k)\geq1$ (Lemma~\ref{lem:throughcycles}); put $Z_1=\rad\Gamma$ and $Z_{n+1}=\ker(\Gamma^{r_n}\to Z_n)$, so $Z_n$ has minimal generator number $r_n$. For a minimal generating set $G$ of $\Omega^{n-1}$ the cover is $A_n=\bigoplus_{g\in G}P_{v(g)}$ with copy-diagonal actions; for a $P_x$-copy $C$ call $e_xC=\Gamma$ its \emph{$\Gamma$-slot}. Three observations, the second and third via Lemma~\ref{lem:throughcycles}: (O1) a generator at a vertex $y\neq x$ has zero $\Gamma$-slot coordinates; (O2) for $q\in P_x$, the $\Gamma$-slot part of $q\cdot w$ is $(e_xq)\gamma$ if $w=\gamma$ lies in the $\Gamma$-slot, and lies in $(x\La y)(y\La x)=0$ if $w$ lies in the $y\La x$-part, $y\neq x$; (O3) since $\rad M=\sum_{\text{arrows}\,a}aM$, the only arrows producing $e_x$-content are loops at $x$, which multiply $\Gamma$-slot coordinates into $\rad\Gamma\cdot(-)$, and arrows $\beta\colon y\to x$ from external in-neighbours, which act on $e_y$-elements whose coordinates on $P_x$-copies lie in $y\La x=0$.

We prove by induction: \emph{$\Omega^n$ has a minimal generating set $G_n$ containing $K_n=\{\tau_1,\dots,\tau_{r_n}\}$ such that (a) each $\tau_j$ is supported on the $\Gamma$-slots of the copies over $K_{n-1}$ (over the canonical generator of $A_1=P_x$ when $n=1$), the tuples $t(\tau_j)\in\Gamma^{r_{n-1}}$ minimally generating $Z_n$; (b) every other element of $G_n$ has zero coordinates on those slots; (c) for every $z\in\Omega^n$ the tuple $t(z)$ of its $\Gamma$-slot coordinates over the $K_{n-1}$-copies lies in $Z_n$.}

\emph{Base.} $\Omega^1=\rad P_x$. The loops exist, are minimal generators (arrows avoid $\rad^2\La$), and their classes minimally generate $\rad\Gamma=Z_1$ because loop words of length $\geq2$ lie in $\rad^2$; (c) is $e_x\rad P_x=\rad\Gamma$. For (b), a generator $h$ with $\Gamma$-slot coordinate $p=\sum q_j\bar\ell_j\in Z_1$ is replaced by $h-\sum q_j\tau_j$; collectively this is a unitriangular change of basis of the top, so minimal generation persists.

\emph{Step.} Let $\varphi\colon A_{n+1}\to\Omega^n$ be the cover and for $z\in A_{n+1}$ let $s(z)\in\Gamma^{r_n}$ collect the $e_x$-parts of its coordinates over the $K_n$-copies. By (O1), (O2) and clause (b) at level $n$, the $K_{n-1}$-$\Gamma$-slot tuple of $\varphi(z)$ equals the pairing of $s(z)$ against the generators $t(\tau_j)$ of $Z_n$; hence $z\in\ker\varphi$ implies $s(z)\in Z_{n+1}$, which is (c). Choose $s^{(1)},\dots,s^{(r_{n+1})}$ minimally generating $Z_{n+1}$ and let $\tau'_j\in A_{n+1}$ carry exactly these $e_x$-coordinates over the $K_n$-copies and zero elsewhere. Since the coordinates lie in $\Gamma$ and each $\tau_t$ is pure, $\varphi(\tau'_j)$ is supported on the $K_{n-1}$-$\Gamma$-slots with tuple the $Z_n$-relation, i.e.\ $\varphi(\tau'_j)=0$: so $\tau'_j\in\Omega^{n+1}$. By (O3) and (c), every element of $\rad\La\cdot\Omega^{n+1}$ has $s$-tuple in $\rad\Gamma\cdot Z_{n+1}$; minimal generators of $Z_{n+1}$ avoid $\rad\Gamma\cdot Z_{n+1}$, so each $\tau'_j$ is a minimal generator, and $\sum c_j\tau'_j\in\rad\Omega^{n+1}$ forces $c_j=0$, so their top classes are independent. Complete to a minimal generating set and normalize the remaining generators as in the base to secure (b). Each $\tau\in K_n$ is a minimal generator at vertex $x$, whence $\dim\Ext^n_\La(S_x,S_x)\geq r_n$.
\end{proof}

\begin{remark}[Relation to the extension conjecture]\label{rem:hanext}
Under the protection hypothesis, Lemma~\ref{lem:throughcycles} makes $\Gamma_x$ loop-generated, so $\Gamma_x\neq k$ forces $\Ext^1(S_x,S_x)\neq0$: Theorem~\ref{thm:protectedcorner} thus furnishes new instances of the extension conjecture of Liu--Morin \cite{LiuMorin,ILP}, with the stronger conclusion of nonvanishing in every degree and the explicit lower bound $\dim\Ext^n_{\Gamma_x}(k,k)$, beyond the known monomial \cite{GSZ} and special biserial \cite{LiuMorin} cases. Han's preprint \cite{HanExt} proposes a proof of the full extension conjecture for elementary algebras by recollements and dg methods; it appears to have remained unpublished, and \cite{CLMS} treat the conjecture as open. The results below are independent of it. If its main theorem is granted, then no vertex of any Gap-A failure carries a loop, a loop gives $\Ext^1(S_z,S_z)\neq0$, hence diagonal Ext in infinitely many degrees, hence $b_n\neq0$ infinitely often, which eliminates the all-infinite configuration a second time and forces the corners of the dumbbell of Theorem~\ref{thm:trichotomy}(2) to be generated by through-cycles.
\end{remark}

\begin{theorem}[$v=3$: the all-infinite case is impossible]\label{thm:trichotomy}
Let $\La$ be a failure of infinite $+$ with $Q$ strongly connected, $|Q_0|=3$, and let $P_\infty$ be the set of vertices whose simple has infinite projective dimension. Then $1\leq|P_\infty|\leq2$. Moreover:
\begin{enumerate}
\item $|P_\infty|=3$ cannot occur: the configuration of Theorem~\ref{thm:allinfinite} does not exist. In particular Corollary~\ref{cor:noloopstratum}(1) is superseded, not merely single-loop corners but the entire stratum is empty.
\item If $P_\infty=\{x,z\}$ then, in the notation of Theorem~\ref{thm:concentration} at $x$, necessarily $z=w$ and the configuration is the \emph{mutual dumbbell}: the arrows between distinct vertices are exactly $u\to x$, $x\to u$, $u\to w$, $w\to u$ (with multiplicities); $x\La w=w\La x=0$; $T_n(x)=\{w\}$ and $T_n(w)=\{x\}$ for $n\gg0$; $\Gamma_x\neq k\neq\Gamma_w$; $u$ carries no loop and $\pd S_u<\infty$; and $u\La w$, $u\La x$ are Tor-rigid against the respective tails as in Theorem~\ref{thm:tail}.
\end{enumerate}
\end{theorem}

\begin{proof}
$P_\infty\neq\emptyset$ since $\gldim\La=\infty$. (1) In the all-infinite configuration, Theorem~\ref{thm:allinfinite} forces at each vertex $i$: $\Gamma_i\neq k$, unique external in-neighbour $i-1$, and $(i-1)\La i=0$ among the Peirce deaths. These are exactly the hypotheses of Theorem~\ref{thm:protectedcorner} at $i$, whose conclusion $\Ext^n(S_i,S_i)\neq0$ for all $n$ contradicts Proposition~\ref{prop:anatomy}. (2) Let $P_\infty=\{x,z\}$. If $z=u$: an arrow $x\to u$ would give $u$ two external in-neighbours (with $w\to u$ from Theorem~\ref{thm:concentration}(2) at $x$), forcing $Z(u)=\emptyset$ against $\pd S_u=\infty$; so $\operatorname{in}(u)=\{w\}$, $Z(u)=\{x\}$, $u\La x=0$, and $\Gamma_x=\Gamma_{\sigma(u)}\neq k$ by Theorem~\ref{thm:concentration}(4) at $u$. Theorem~\ref{thm:protectedcorner} fires at $x$: contradiction. If $z=w$ and the arrows into $w$ come from $x$: then $x\La w=0$ (the dead zone of $x$) is precisely the protection hypothesis at $w$, and $\Gamma_w=\Gamma_{\sigma(x)}\neq k$ by Theorem~\ref{thm:concentration}(4) at $x$; the theorem fires at $w$: contradiction. Hence the arrows into $w$ come from $u$ alone; an arrow $x\to w$ is then excluded, and strong connectivity forces $x\to u$. Theorem~\ref{thm:concentration} at $w$ now gives $Z(w)=\{x\}$, $w\La x=0$, and $\Gamma_x\neq k$; at $x$ it gives $\Gamma_w\neq k$; the strong no-loop theorem \cite{ILP} forbids loops at $u$ since $\pd S_u<\infty$; Tor-rigidity is Theorem~\ref{thm:tail}(2) applied at $x$ and at $w$.
\end{proof}

\begin{remark}[The $v=3$ frontier after this section]\label{rem:frontier2}
Gap A for three strongly connected vertices now rests on exactly two configurations. \emph{(i) One infinite vertex}: the shape of Theorem~\ref{thm:concentration}, with the contrapositive of Theorem~\ref{thm:protectedcorner} adding: either $\Gamma_x=k$ or $u\La x\neq0$; and, since $\Gamma_w\neq k$ is forced, some in-neighbour $y$ of $w$ has $y\La w\neq0$. \emph{(ii) The mutual dumbbell} of Theorem~\ref{thm:trichotomy}(2). Observe that the dumbbell contains a simple of finite projective dimension at a loop-free hub, so the Hattori--Stallings trace methods of Igusa and Igusa--Liu--Paquette \cite{Igusa90,ILP}, which require finite projective resolutions to exist, become available; the all-infinite case never offered this. On the computational side, exact-arithmetic minimal-resolution computation extends to multi-loop corners and to both configurations; this is the immediate next step.
\end{remark}

\subsection*{Two consequences of the protected corner}

The protected corner theorem yields, on the one hand, a sufficient criterion for Gap A parallel to the corner-projectivity criterion of Corollary~\ref{cor:cornerscc} but resting on the \emph{dead-return} hypothesis rather than on projectivity, and, on the other hand, a quantitative lower bound on the Betti numbers of $S_x$ that is already exponential at a two-loop corner.

\begin{corollary}[Protected-corner criterion for Gap A]\label{cor:protectedgapA}
Gap A holds for $\La=kQ/I$ whenever every strongly connected component $\CC$ of $Q$ with $\gldim\La_\CC=\infty$ contains a vertex $x$ with $\Gamma_x=e_x\La_\CC e_x\neq k$ such that $y\La_\CC x=0$ for every in-neighbour $y\neq x$ of $x$ in $Q_\CC$, that is, no path of $\La_\CC$ returns from $x$ to a vertex pointing at $x$. Such a vertex satisfies the extension conjecture inside $\La_\CC$ with the explicit bound
$\dim\Ext^n_{\La_\CC}(S_x,S_x)\geq\dim\Ext^n_{\Gamma_x}(k,k)\geq1$ of Theorem~\ref{thm:protectedcorner}, so $\La_\CC$ is of infinite $+$ global dimension and Theorem~\ref{thm:scc} lifts this to $\La$.
\end{corollary}

\begin{proof}
By Theorem~\ref{thm:scc}(3) it suffices that every $\La_\CC$ of infinite global dimension be of infinite $+$ global dimension. Since $\La_\CC\cong kQ_\CC/I_\CC$ is elementary (Lemma~\ref{lem:scc-corner}) and the displayed vertex is protected inside it, this is Theorem~\ref{thm:protectedcorner} applied to $\La_\CC$.
\end{proof}

This is genuinely disjoint from Corollary~\ref{cor:cornerscc}: there the corner algebra must see $e_x\La e_y$ \emph{projective} over $\Gamma_x$ for the relevant $y$, here it must see the return maps $y\La x$ \emph{die}. Neither hypothesis implies the other; both force the same conclusion through the syzygy-return criterion.

\begin{corollary}[Exponential Betti growth at a radical-square-zero corner]\label{cor:expbetti}
Let $x$ be protected as in Theorem~\ref{thm:protectedcorner} ($\Gamma_x\neq k$, and $y\La x=0$ for every external in-neighbour $y$), and suppose $\rad^2\Gamma_x=0$; write $\ell:=\dim_k\rad\Gamma_x$, the number of loops at $x$. Then
\[
\dim_k\Ext^n_\La(S_x,S_x)\;\geq\;\ell^{\,n}\qquad\text{for all }n\geq0 .
\]
In particular, at a protected corner with $\ell\geq2$ loops and all length-two loop-words zero, the Betti numbers of $S_x$ grow at least exponentially with base $\ell$; a single such vertex already forces $\gldim\La=\infty$ with exponentially many extensions in every degree.
\end{corollary}

\begin{proof}
By Lemma~\ref{lem:throughcycles} the protection hypothesis makes $\Gamma_x$ loop-generated, so with $\rad^2\Gamma_x=0$ one has $\Gamma_x=k\oplus\rad\Gamma_x$ with $\rad\Gamma_x$ spanned by the $\ell$ loops and annihilated by $\rad\Gamma_x$; thus $\rad\Gamma_x\cong k^{\ell}$ is semisimple over $\Gamma_x$. In the minimal $\Gamma_x$-resolution of $k$ every syzygy is semisimple: $\Omega^0=k$, and if $\Omega^n\cong k^{d}$ then its projective cover is $\Gamma_x^{d}$ with kernel $\Omega^{n+1}=\rad\Gamma_x\cdot\Gamma_x^{d}=(\rad\Gamma_x)^{\oplus d}\cong k^{\ell d}$. Hence $\dim_k\Omega^n_{\Gamma_x}(k)=\ell^{n}$ and, the syzygies being semisimple, $\dim_k\Ext^n_{\Gamma_x}(k,k)=\ell^{n}$ (equivalently $\Ext^\bullet_{\Gamma_x}(k,k)$ is the tensor algebra on $(\rad\Gamma_x)^\ast$). Theorem~\ref{thm:protectedcorner} gives $\dim\Ext^n_\La(S_x,S_x)\geq\dim\Ext^n_{\Gamma_x}(k,k)=\ell^{n}$.
\end{proof}

\begin{example}[A protected two-loop corner attaining the exponential bound]\label{ex:twoloop}
Let $Q$ have vertices $x,u$, two loops $\alpha,\beta$ at $x$ and one arrow $c\colon u\to x$, and set
\[
\La=kQ\big/\big\langle\alpha^2,\ \beta^2,\ \alpha\beta,\ \beta\alpha\big\rangle,\qquad \dim_k\La=7 .
\]
Then $\Gamma_x=e_x\La e_x=k\langle\alpha,\beta\rangle/(\alpha,\beta)^2$ is the three-dimensional radical-square-zero corner ($\ell=2$), and there is no path $x\to u$, so $u\La x=0$ and $x$ is protected; three arrows $\alpha,\beta,c$ enter $x$, so $\La$ is not special biserial (a special biserial algebra has at most two arrows entering, and two leaving, each vertex). Corollary~\ref{cor:expbetti} gives $\dim\Ext^n_\La(S_x,S_x)\geq2^n$. This is the non-uniserial counterpart of the uniserial single-loop corners of $\La(m;d)$, where $\dim\Ext^n_{\Gamma_x}(k,k)=1$ and the growth of $\dim\Ext^n_\La(S_x,S_x)$ is only linear (Remark~\ref{rem:certificate}).
\end{example}

\section{Hattori--Stallings traces on the dumbbell}\label{sec:dumbbell}

Throughout this section $\La$ is a dumbbell failure as in Theorem~\ref{thm:trichotomy}(2), with threshold $N$, hub $u$, infinite vertices $x,w$, and
\[
p=\dim x\La u,\quad q=\dim u\La x,\quad r=\dim u\La w,\quad s=\dim w\La u,\quad
C_x:=u\La x\cdot x\La u,\quad C_w:=u\La w\cdot w\La u .
\]
All four hub numbers are $\geq1$ and $\Gamma_x\neq k\neq\Gamma_w$. In this section, we exploit what no earlier configuration contained, namely a simple module of finite projective dimension, and we prove that this finiteness propagates through the whole structure of the dumbbell. The trace technology below descends from Hattori \cite{Hattori} and Stallings \cite{Stallings} through Lenzing \cite{Lenzing}, Igusa \cite{Igusa90}, Igusa--Liu--Paquette \cite{ILP} and Han \cite{HanExt}, who extended Hattori--Stallings traces to perfect complexes and bimodules; Lemma~\ref{lem:transfer} is a two-resolution transfer variant adapted to the situation where one of the resolutions is infinite.

\begin{proposition}[Finiteness at the hub]\label{prop:flanks}
$\pd S_u\leq N-1$ and $\operatorname{inj.dim}S_u\leq N-1$.
\end{proposition}

\begin{proof}
For $n\geq N$ the one-way law gives $\Ext^n(S_u,S_v)\neq0\Rightarrow u\La v=0$; but $u\La x,u\La w\neq0$ (arrows) and $e_u\in u\La u$, so $\Ext^N(S_u,S_v)=0$ for every $v$, the $N$-th syzygy has no top, and $\Omega^N(S_u)=0$. Dually $\Ext^n(S_x,S_u)=\Ext^n(S_w,S_u)=0$ for $n\geq N$ since the tops of those syzygies sit at $w$, resp.\ $x$; and $\Ext^n(S_u,S_u)=0$ for $n>\pd S_u$.
\end{proof}

\begin{proposition}[Eventual vanishing of Hochschild homology]\label{prop:hhvanish}
Let $\La$ be \emph{any} failure of infinite $+$ (any number of vertices). Then $\mathrm{HH}_n(\La)=0$ for all $n\gg0$, and $\sum_i(-1)^i\dim\mathrm{HH}_i(\La)=v$, in every characteristic. Thus every Gap-A failure is a Han counterexample in the strongest sense, with the Euler characteristic pinned at $v$.
\end{proposition}

\begin{proof}
Not infinite $+$ means $\sum_n\dim \mathrm{HH}^\tau_n<\infty$ \cite[Thm.~5.5]{CLMS}, so $\mathrm{HH}^\tau_n=0$ for $n\geq N_1$. By Proposition~\ref{prop:window}, $\dim\mathrm{HH}^\tau_n=\operatorname{tr}S_{n+1}+(-1)^n(v-\chi_{n-1})$ with $\chi_m=\sum_{i\leq m}(-1)^i\dim\mathrm{HH}_i$, and $\operatorname{tr}S_{n+1}=0$ for $n\geq N$. Hence $\chi_{n-1}=v$ for all $n\geq\max(N,N_1)$, so consecutive partial sums agree and $\dim\mathrm{HH}_n=0$ there; the stabilized sum is $v$.
\end{proof}

\begin{theorem}[Gluing]\label{thm:gluing}
In the dumbbell:
\begin{enumerate}
\item $x\La u\cdot u\La w=0=w\La u\cdot u\La x$, and $\rad\Gamma_u=C_x+C_w$ with $C_xC_w=C_wC_x=0$.
\item $J_x:=\La e_x\La=\Gamma_x\oplus x\La u\oplus u\La x\oplus C_x$ and $J_w:=\La e_w\La$ satisfy $J_xJ_w=J_wJ_x=0$; the ideal $Z:=J_x\cap J_w=C_x\cap C_w$ is central, $Z\cdot\rad\La=\rad\La\cdot Z=0$, and $Z\cong S_u^{\dim Z}$ as a one-sided module.
\item $\La/Z\;\cong\;\La_x\times_k\La_w$, the fibre product over the augmentations at $u$ of the two \emph{two-vertex} algebras $\La_x:=\La/J_w$ (quiver $u\rightleftarrows x$ plus loops at $x$) and $\La_w:=\La/J_x$. Moreover, no nonzero cyclic path class meets both $x$ and $w$, so $\mathrm{HH}_0(\La)=k^3\oplus H_x\oplus H_w$ with $H_x,H_w$ spanned by one-sided necklaces (overlapping at most in the classes of $Z$).
\end{enumerate}
\end{theorem}

\begin{proof}
(1) $x\La u\cdot u\La w\subseteq x\La w=0$ and symmetrically. A cycle at $u$ visiting both $x$ and $w$ contains a segment from an $x$-visit to a $w$-visit, whose class lies in $w\La x=0$; a cycle visiting only $x$ decomposes at its $u$-returns into products of primitive excursions, each in $u\La x\cdot\Gamma_x\cdot x\La u\subseteq C_x$, and $C_x$ is multiplicatively closed since $x\La u\cdot u\La x\subseteq\Gamma_x$. Finally $C_xC_w\subseteq u\La x\,(x\La u\cdot u\La w)\,w\La u=0$.

(2) $\La e_x\La=(\Gamma_x\oplus u\La x)(\Gamma_x\oplus x\La u)$ gives the displayed Peirce shape; $J_xJ_w\subseteq\La(e_x\La e_w)\La=\La\,x\La w\,\La=0$. The intersection is the common $(u,u)$-part $C_x\cap C_w$. For $z\in Z$: $z\cdot u\La x\subseteq C_w\cdot u\La x\subseteq u\La w\cdot(w\La u\cdot u\La x)=0$, $z\cdot u\La w\subseteq C_x\cdot u\La w=0$, $z\cdot\rad\Gamma_u\subseteq C_xC_w+C_wC_x=0$, and symmetrically on the left; $z$ commutes with the idempotents since $z=e_uze_u$. So $z$ is central with $z\,\rad=\rad\,z=0$, and $\rad\La\cdot Z=0$ makes $Z$ semisimple at $u$.

(3) The map $\La\to\La/J_x\times\La/J_w$ has kernel $Z$ and image the pairs agreeing in $\La/(J_x{+}J_w)=k\bar e_u$; comparing dimensions ($\dim J_x+\dim J_w=\dim\La-1+\dim Z$) shows the image is exactly the fibre product. A cycle through both $x$ and $w$ vanishes as in part (1).
\end{proof}

\begin{corollary}[Hub splitting]\label{cor:hubsplit}
$R_x:=x\La u\oplus C_x$ and $R_w:=w\La u\oplus C_w$ are $\La$-submodules of $\rad P_u$ with $R_x+R_w=\rad P_u=\Omega^1(S_u)$ and $R_x\cap R_w=Z$, giving an exact sequence
$0\to S_u^{\dim Z}\to R_x\oplus R_w\to\Omega^1(S_u)\to0$.
Consequently $\pd R_x,\;\pd R_w\leq N-1$; moreover $J_wR_x=0$ and $J_xR_w=0$, so $R_x$ is a $\La_x$-module and $R_w$ a $\La_w$-module of finite projective dimension over $\La$.
\end{corollary}

\begin{proof}
Submodule and intersection claims are Peirce bookkeeping with Theorem~\ref{thm:gluing}(1); e.g.\ $w\La u\cdot C_x\subseteq(w\La u\cdot u\La x)x\La u=0$ and $x\La u\cdot C_x\subseteq\Gamma_x\cdot x\La u\subseteq x\La u$. The sequence is the Mayer--Vietoris sequence of the two submodules. Since $\pd Z=\pd S_u$ and $\pd\Omega^1(S_u)=\pd S_u-1$, the middle term has $\pd\leq\pd S_u\leq N-1$. Finally $J_wR_x\subseteq J_wJ_x=0$.
\end{proof}

\begin{proposition}[The tails are one-sided; the $u$-blind reduction]\label{prop:ublind}
For $n\geq N$, $\Omega^n(S_x)$ is a $\La_w$-module, $P_w=\La_w\bar e_w$, and the tail of the minimal $\La$-resolution of $S_x$ is the minimal $\La_w$-resolution of $M_x:=\Omega^N(S_x)$. The module $M_x$ satisfies $\pd_{\La_w}M_x=\infty$ and $\Ext^j_{\La_w}(M_x,S_u)=0$ for all $j\geq0$ (\emph{$u$-blind}). Symmetrically $M_w:=\Omega^N(S_w)$ is $u$-blind over $\La_x$. Hence each glued factor is a two-vertex algebra of infinite global dimension carrying a $u$-blind module of infinite projective dimension.
\end{proposition}

\begin{proof}
$\Omega^n(S_x)$ is generated at $w$ (Theorem~\ref{thm:trichotomy}), so $\Omega^n=\La\,e_w\Omega^n$ and $J_x\Omega^n=(J_xe_w)\La$-combinations $=0$ since $J_x$ has no $e_w$-column. $\La_w\bar e_w=\La e_w/J_xe_w=P_w$, so $\La$- and $\La_w$-covers of these modules coincide, and the tops-at-$w$ statement of the one-way law becomes $\Ext^j_{\La_w}(M_x,S_u)=0$.
\end{proof}

\begin{theorem}[Head extraction]\label{thm:head}
Applying the exact Peirce functors to the finite head $0\to\Omega^N(S_x)\to P_{N-1}\to\cdots\to P_0\to S_x\to0$ and to the finite resolution $Q_\bullet$ of $S_u$ yields finite exact complexes of corner modules:
\begin{enumerate}
\item $0\to e_xP_{N-1}\to\cdots\to e_xP_0\to k\to0$ over $\Gamma_x$, with terms in $\operatorname{add}(\Gamma_x\oplus x\La u)$;
\item the mirror at $w$, with terms in $\operatorname{add}(\Gamma_w\oplus w\La u)$;
\item $0\to e_uQ_d\to\cdots\to e_uQ_0\to k\to0$ over $\Gamma_u$, with terms in $\operatorname{add}(\Gamma_u\oplus u\La x\oplus u\La w)$.
\end{enumerate}
\end{theorem}

\begin{proof}
$\Omega^N(S_x)$ is generated at $w$, so $e_x\Omega^N(S_x)=e_x\La e_w\cdot e_w\Omega^N=x\La w\cdot(-)=0$; exactness is preserved by $e_x(-)$, and $e_xP_x=\Gamma_x$, $e_xP_u=x\La u$, $e_xP_w=x\La w=0$. Same at $w$ and at $u$ (where $e_uS_u=k$ and the complex is the whole finite resolution).
\end{proof}

\begin{corollary}\label{cor:nonproj}
$x\La u$ is \emph{not} projective as a left $\Gamma_x$-module and $w\La u$ is \emph{not} projective as a left $\Gamma_w$-module; if $\Gamma_u\neq k$, at least one of $u\La x$, $u\La w$ is not projective over $\Gamma_u$. Equivalently, in the singularity categories: $[x\La u]$ classically generates $D_{\mathrm{sg}}(\Gamma_x)$ and $[w\La u]$ generates $D_{\mathrm{sg}}(\Gamma_w)$.
\end{corollary}

\begin{proof}
If $x\La u$ were projective, Theorem~\ref{thm:head}(1) would be a finite projective resolution of $k$ over the local algebra $\Gamma_x$, forcing $\Gamma_x=k$ by the socle argument of Lemma~\ref{lem:throughcycles}, against Theorem~\ref{thm:trichotomy}. Same at $w$ and at $u$. For the last statement: a finite exact complex with terms in $\operatorname{add}(\Gamma\oplus V)$ augmented to $k$ says $k\in\operatorname{thick}(\Gamma\oplus V)$ in $D^b(\Gamma)$, i.e.\ $[k]\in\operatorname{thick}[V]$ in $D_{\mathrm{sg}}(\Gamma)$ \cite{Buchweitz,Orlov}; since $k$ generates $D_{\mathrm{sg}}$ of an artinian local algebra, $[V]$ does too.
\end{proof}

\begin{proposition}[Cartan rigidity]\label{prop:cartanD}
Order the vertices $(x,u,w)$; the Cartan matrix of Definition~\ref{def:matrices}, $C_{ab}=\dim_k b\La a$, is then
$C=\left(\begin{smallmatrix}\gamma_x&q&0\\ p&\gamma_u&s\\ 0&r&\gamma_w\end{smallmatrix}\right)$ with $\gamma_\bullet=\dim\Gamma_\bullet$. Then $\det C\neq0$,
\[
\det C\cdot(\chi_x,\chi_u,\chi_w)=(-p\gamma_w,\;\gamma_x\gamma_w,\;-s\gamma_x),\qquad \chi_v:=\textstyle\sum_i(-1)^i\dim\Ext^i(S_u,S_v),
\]
and, using $\operatorname{inj.dim}S_u<\infty$ on the opposite side, $\det C$ divides $\gcd(\gamma_x\gamma_w,\,p\gamma_w,\,q\gamma_w,\,s\gamma_x,\,r\gamma_x)$. In particular $\chi_u=\gamma_x\gamma_w/\det C$ is a nonzero integer.
\end{proposition}

\begin{proof}
The finite resolution of $S_u$ gives $C^{\mathsf T}\chi=e_u$ in $\mathbb Z^3$ (apply each $e_a$ and count dimensions, using $\dim e_a P_z=C_{za}$). Multiplying by the adjugate, $\det C\cdot\chi=\operatorname{adj}(C^{\mathsf T})e_u$, whose entries are the displayed cofactors. If $\det C=0$ then $\gamma_x\gamma_w=0$, absurd. The opposite algebra has Cartan matrix $C^{T}$ and $DS_u$ has finite projective dimension over it, giving $\det C\mid q\gamma_w$ and $\det C\mid r\gamma_x$ from the transposed cofactors.
\end{proof}

\begin{lemma}[Transfer trace through the hub]\label{lem:transfer}
Let $\xi\in x\La u$, $\eta\in u\La x$, and let $Q_\bullet\to S_u$ be the minimal resolution, $d=\pd S_u$. Choose chain maps $h\colon P_\bullet(S_x)\to Q_\bullet$ and $g\colon Q_\bullet\to P_\bullet(S_x)$ lifting the right multiplications $\rho_\xi\colon\La e_x\to\La e_u$ and $\rho_\eta\colon\La e_u\to\La e_x$ (both cover zero maps of simples). Then in $\mathrm{HH}_0(\La)$
\[
[\xi\eta]\;=\;[\eta\xi]\;=\;\sum_{i=1}^{d}(-1)^{i+1}\operatorname{tr}_{\mathrm{HS}}(h_ig_i)\;=\;\sum_{i=1}^{d}(-1)^{i+1}\operatorname{tr}_{\mathrm{HS}}(g_ih_i),
\]
a \emph{finite} sum although $P_\bullet(S_x)$ is infinite. The same holds at $w$. In particular, if $x$ carries no loops, every class of $\rad\Gamma_x$ in $\mathrm{HH}_0$ is such a finite alternating trace through the hub's resolution.
\end{lemma}

\begin{proof}
$hg$ is a chain self-map of $Q_\bullet$ with $(hg)_0=\rho_{\eta\xi}$, lifting the zero endomorphism of $S_u$, hence null-homotopic: $(hg)_i=d_{i+1}s_i+s_{i-1}d_i$. Applying $\operatorname{tr}_{\mathrm{HS}}$ and the symmetry $\operatorname{tr}_{\mathrm{HS}}(AB)=\operatorname{tr}_{\mathrm{HS}}(BA)$, the alternating sum telescopes: $\sum_{i=0}^d(-1)^i\operatorname{tr}_{\mathrm{HS}}((hg)_i)=0$. The $i=0$ term is $[\eta\xi]=[\xi\eta]$ (rotation), and $(hg)_i=0$ for $i>d$. The last equality is trace symmetry degreewise; the final claim holds because loopless $x$ has $\rad\Gamma_x=x\La u\cdot u\La x$ by Theorem~\ref{thm:gluing}(1).
\end{proof}

\begin{remark}[Concluding remarks on the dumbbell]\label{rem:dumbbellstatus}
We have not obtained a contradiction: the dumbbell survives this section, although in a considerably reduced form. It is (up to the central socle ideal $Z$) a vertex-gluing of two \emph{two-vertex} algebras, each of infinite global dimension, each carrying a $u$-blind infinite tail (Proposition~\ref{prop:ublind}); both linking modules $x\La u$, $w\La u$ are forced non-projective over nontrivial corners while, if a corner is uniserial, the argument of Corollary~\ref{cor:uniserial} applied to the opposite tail forces the \emph{other-sided} linking module ($u\La x$ over $\Gamma_x$, $u\La w$ over $\Gamma_w$) to be free, an explicitly asymmetric matrix condition; the Cartan determinant is nonzero and divides the five products of Proposition~\ref{prop:cartanD}; and $\mathrm{HH}_n(\La)=0$ for $n\gg0$ with $\chi=3$ (Proposition~\ref{prop:hhvanish}). The next step is now well defined: for a vertex-gluing, the reduced Hochschild complex over $E=k^3$ splits into the complexes of the two sides plus \emph{mixed necklaces} alternating through $u$, whose homology is built from $\operatorname{Tor}^{\La_x}_\bullet(S_u,S_u)$ and $\operatorname{Tor}^{\La_w}_\bullet(S_u,S_u)$; Proposition~\ref{prop:hhvanish} forces the mixed part to die in high degrees, while Proposition~\ref{prop:ublind} keeps both sides homologically infinite. Whether these two demands are compatible is a concrete computation, by hand for the necklace differential, or by exact-arithmetic computation extended to two-vertex gluings, and an incompatibility would eliminate the dumbbell, collapsing Gap A at $v=3$ to the single-infinite configuration. Two placements in the literature sharpen the picture. The dumbbell is a null-square-type extension of $\La_x\times\La_w$ in the sense of \cite{CRS}, but Corollary~\ref{cor:nonproj} forces its connecting bimodules to be non-projective over nontrivial corners: it sits exactly outside the null-square \emph{projective} class for which Han's conjecture is known to propagate, and outside the bounded-extension classes of \cite{CLMSsplit,CLMSbounded}. And the mixed-necklace computation belongs to the Poincar\'e-series tradition for fibre products of local rings going back to Dress and Kr\"amer \cite{DressKramer}; see also \cite{NassehSW} for Tor-vanishing over fibre products.
\end{remark}

\subsection*{Disclosure}
During the preparation of this work, the author used Claude, Anthropic's Fable 5 model, for deep research, formulation of theorems, and drafts of their proofs. The author reviewed and edited the output as needed and takes full responsibility for the content of the published article.

\end{document}